\documentclass[10pt, tikz]{amsart}

\usepackage{txfonts}
\usepackage[T1]{fontenc}
\usepackage{fancyhdr}

\usepackage{geometry}
\usepackage{amsmath,amsthm,amssymb,hyperref, mdframed}
\usepackage{mathrsfs} 
\usepackage{tikz}
\usepackage[all]{xy}
\usepackage{subfig}
\usepackage{tikz}
\usepackage{subfig}
\usepackage{tabularx}
\usepackage{enumitem}
\usepackage{amsmath,amsthm,amscd}
\usepackage{amssymb,amsfonts}
\usepackage{fancyhdr} 
\usepackage{pgf,tikz}
\usepackage{tikz,tikz-3dplot}
\usepackage{caption}
\usepackage{tikz-cd}
\usepackage{tikzscale}
\usepackage{rotating}
\usepackage{xcolor}
\usepackage{lscape}
\usepackage{xspace}
\usepackage{xfrac}
\usepackage{blindtext}
\usepackage[ruled,vlined]{algorithm2e}
\usepackage[export]{adjustbox}
\usepackage[dvipsnames]{pstricks}
\usepackage{pst-solides3d}
\usepackage{color}
\usepackage{hyperref}
\hypersetup{
    colorlinks=true, 
    linktoc=all,     
    linkcolor=blue,  
    citecolor=blue,
    filecolor=blue,
    urlcolor=blue
}

\usetikzlibrary{arrows}
\usetikzlibrary{decorations.markings}
\usetikzlibrary{decorations.markings}
\usetikzlibrary{backgrounds,shapes}
\usetikzlibrary{decorations.markings}
\usetikzlibrary{patterns}
\usetikzlibrary{calc,perspective}

    \pgfdeclarelayer{background} 
    \pgfdeclarelayer{foreground}
    \pgfsetlayers{background,main,foreground}

\definecolor{aliceblue}{rgb}{0.9, 0.95, 1.0}

\definecolor{violet1}{RGB}{69, 14, 71}
\definecolor{color2}{RGB}{162, 165, 149}
\definecolor{color3}{RGB}{180, 162, 138}

\definecolor{cof}{RGB}{219,144,71}
\definecolor{pur}{RGB}{186,146,162}
\definecolor{greeo}{RGB}{91,173,69}
\definecolor{greet}{RGB}{52,111,72}

\newcommand{\cm}[1]{{ \textcolor{blue}{#1}} }

\newcommand{\mytriangle}[4] 
{
  \coordinate (center) at ($1/3*(#1)+1/3*(#2)+1/3*(#3)$);
  \coordinate (m12)    at ($(#1)!0.5!(#2)$);
  \coordinate (m13)    at ($(#1)!0.5!(#3)$);
  \coordinate (m23)    at ($(#2)!0.5!(#3)$);
  \draw[fill=pink] (center) -- (m12) -- (#1) -- (m13) -- cycle;
  \draw[fill=red] (center) -- (m12) -- (#2) -- (m23) -- cycle;
  \draw[fill=purple] (center) -- (m13) -- (#3) -- (m23) -- cycle;
  \draw[thick,fill=black,fill opacity=#4]  (#1) -- (#2)  -- (#3) -- cycle;
}

\numberwithin{equation}{section}

\newcommand\Z{{\mathbb Z}}

\newcommand{\TT}{\mathcal{T}}

\newcommand{\C}{{\mathbb C}}

\newcommand{\pslc}{{\mathrm{PSL}(2,\mathbb{C})}}
\newcommand{\pslr}{{\mathrm{PSL}(2,\,\mathbb{R})}}
\newcommand{\pglr}{{\mathrm{PGL}(2,\,\mathbb{R})}}

\newcommand{\rs}{\mathbb{C}\mathrm{\mathbf{P}}^1}

\makeatletter
\def\namedlabel#1#2{\begingroup
   \def\@currentlabel{#2}%
   \label{#1}\endgroup
}
\makeatother

\theoremstyle{plain}                    
\newtheorem{thm}{Theorem}[section]
\newtheorem*{khthm}{Killing-Hopf's Theorem}
\newtheorem*{pthm}{Poincar\'e's Theorem}
\newtheorem*{ethm}{Euler's observation}

\newtheorem{lem}[thm]{Lemma}
\newtheorem{prop}[thm]{Proposition}
\newtheorem{cor}[thm]{Corollary}

\theoremstyle{definition}
\newtheorem{defn}[thm]{Definition}

\newtheorem{ex}[thm]{Example}

\newtheorem{rmk}[thm]{Remark}

\title[Tessellations of surfaces]{Tessellations and triangulation of surfaces}

\author{Gianluca Faraco}

\address{Dipartimento di Matematica e Applicazioni U5, Universita` degli Studi di Milano-Bicocca, Via Cozzi 55, 20125 Milano, Italy}
\email{gianluca.faraco@unimib.it}
\email{gianluca.faraco.math@gmail.com}
\date{\today}

\begin{document}

\keywords{}
\subjclass[]{}%
\dedicatory{}

\begin{abstract}
    A tessellation or tiling is a collection of sets, called tiles, that cover a plane without gaps and overlaps. The present note is an invitation to get to know the beauty and majesty of tessellations and triangulation of orientable surfaces. 
\end{abstract}

\maketitle
\tableofcontents

\section{Introduction}

\noindent In the present survey, we aim to introduce the reader to the topic of \textit{tessellations}. In the real life, we come across examples of tessellations every day: a tiled floor with square tiles is certainly the most typical example, although not the only one. \textit{Tessellate} found its origin from the Latin word "\textit{tessellatus}", past participle of "\textit{tessellare}", which literally means to pave with "\textit{tessellae}", plural of the Latin word "\textit{tessella}", which means small square. Tessella corresponds to the modern term \textit{tile} and a tessellation is commonly known also as \textit{tiling}, \textit{paving} or \textit{mosaic}, see \cite{GS}. Despite its meaning, a tile is not necessarily a square but it can be any polygon, possibly non-regular. More generally, a tile could even be somewhat fancy, not necessarily a polygon.

\begin{figure}[!ht]
  \centering
  \includegraphics[scale=0.0625]{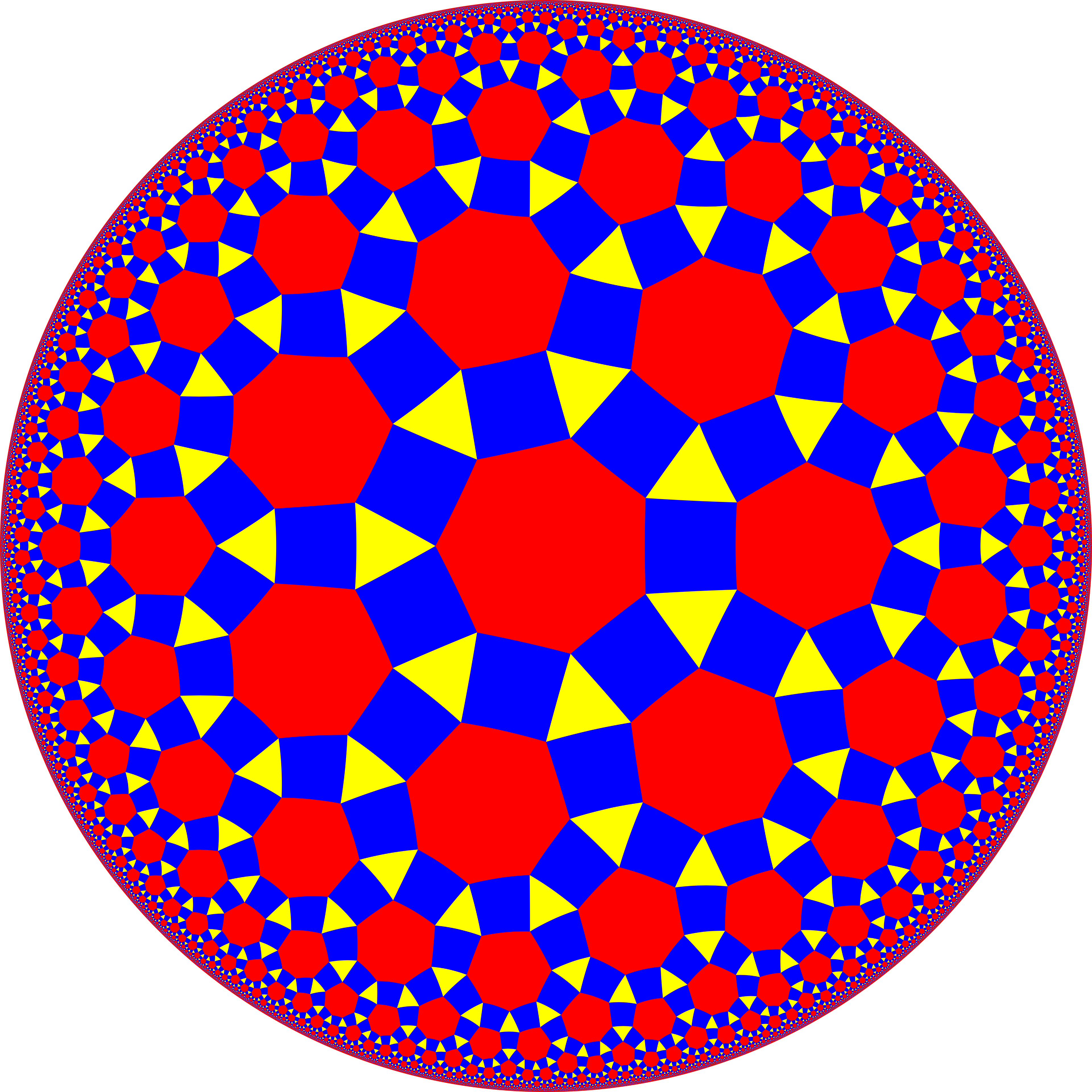}
  \caption{Rhombitrihexagonal tessellation of the Poincar\'e disc.}
  \label{rhomb}
\end{figure}

\noindent Although tessellations have been known since the ancient times, \textit{e.g.} Sumerians used them in building wall decorations formed by patterns of clay tiles, the first pioneering investigation had been done by Johannes Kepler around the seventeenth centenary. In his \textit{Harmonices Mundi}, Kepler studied regular and semi-regular tessellations and he was the first to describe the structure of the honeycomb, which turns out to be one of the three possible tessellations of the Euclidean plane. 

\smallskip

\noindent The beauty and majesty of tessellations, however, can be fully appreciated in the classical art in which they have been extensively used for mosaics or paving floors. An example of such paving can be found at the Archaeological Museum of Seville in Spain where the floor is tiled by using square, triangle and hexagon prototiles. Figure \eqref{rhomb} above gives an idea of what such paving, called \textit{rhombitrihexagonal tiling}, looks like in the case of the Poincar\'e disk. Here we will focus on a mathematical exposition of tessellations of the sphere, the Euclidean plane and hyperbolic plane. We begin with the following question:

\subsection{What is a tessellation?} By a tessellation $\mathcal{T}$, we mean a subdivision of a geometric surface $S$ into non-overlapping polygons, \textit{i.e.} interiors are disjoint, called \textit{tiles}. By a \textit{geometric surface}, we mean a topological surface equipped with one of three $2-$dimensional geometries: spherical, euclidean or hyperbolic. A generic tessellation $\mathcal{T}$ does not need to exhibit any periodic pattern and any such tessellation is commonly known as \textit{aperiodic}. Penrose's tilings are the best-known examples of aperiodic tilings, see \cite{PR}. These tilings are tiled by using two different prototiles and it was a longstanding question whether aperiodic tilings tiled by a single prototile exist. These are sometimes called aperiodic monotiles. A positive answer has been recently given in \cite{SMKGS} for topological disk tiles where the authors exhibit a continuum of combinatorially equivalent aperiodic polygons. A symmetry for a tiling $\TT$ is an isometry of $S$ that preserves the tiling structure \textit{i.e.} it takes vertices, edges and tiles of the tiling to themselves. A tessellation $\mathcal{T}$ is said to be \textit{symmetric} if any tile can be mapped onto any other one by an isometry which maps the whole $\mathcal{T}$ to itself. The basic idea is that tiles of a symmetric tiling, say $\mathcal{T}$, cannot be distinguished and $\mathcal{T}$ looks like just the same from the point of view of any tile. For a symmetric tiling all tiles are pairwise congruent and hence there is only one prototile. A symmetric tiling is called \textit{regular} is the prototile is a regular polygon.

\begin{figure}[!ht]
  \centering
  \includegraphics[scale=0.125]{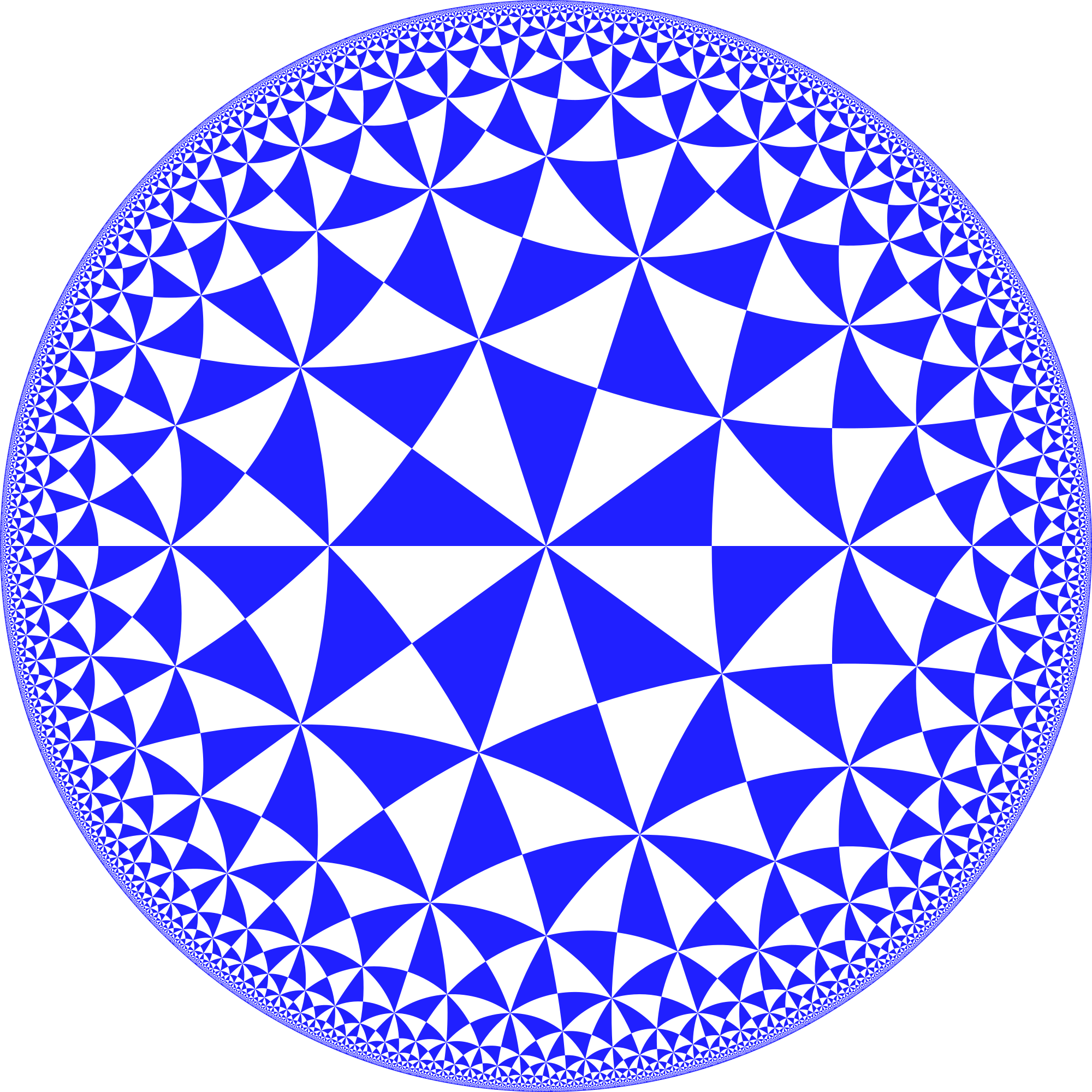}
  \caption{Triangular tessellation of the Poincar\'e disk where each tile is isometric to a hyperbolic triangle having inner angles of magnitude $\Big(\frac{\pi}{2},\,\frac{\pi}{4},\,\frac{\pi}{5}\Big)$. In Section \S\ref{tg2}, we shall see that this tessellation corresponds to the Triangle group $\Delta(2,4,5)$.}
  \label{tg542}
\end{figure}

\noindent Figure \eqref{tg542} depicts an example of regular tilings of the Poincar\'e disk $\mathbb D^2$; we shall discuss, in details, these kind of tessellation in Section \S\ref{tg}. The tessellation shown in Figure \eqref{rhomb} is not regular because there are three different kinds of prototiles. In the present survey, we will mainly focus on symmetric tessellations with short digressions about semi-regular tessellations in \S\ref{srt} and \S\ref{ssec:semiplatsol}. We now focus on another aspect of tessellations, namely the edge-to-edge property. A tiling $\mathcal{T}$ is said to be \textit{edge-to-edge} if the mutual relation of any pair of tiles is one of the following:
\begin{itemize}
    \item they are disjoint (have no point in common);
    \item they have precisely one common point which is a vertex of each of the polygons; or
    \item they share a segment that is an edge of each of the two polygons.
\end{itemize}

\noindent For edge-to-edge tessellations, any point of the plane which is a vertex of some tile is also the vertex of every other tile to which it belongs. In the present note, we will focus on edge-to-edge tilings. The curious reader may consult \cite[Section \S4]{GS} for tilings that are not edge-to-edge. The most familiar edge-to-edge regular tessellations are those of $\mathbb E^2$, the Euclidean plane. In order to give to the reader some concrete examples to think about, we shall describe them in more detail in the forthcoming subsection \S\ref{step}.

\subsection{Symmetric tessellations of the Euclidean plane}\label{step} Before Theaetetus of Athens (c. 417 – c. 369 BC) classified the Platonic solids, we shall introduce them in Section \S\ref{platsol}, the possible edge-to-edge tessellations of the Euclidean plane were already discovered. There are only three regular tessellations of the Euclidean plane. In other words, there are only three regular polygons which, on their own right, can tile the whole Euclidean plane like a mosaic. With a heuristic reasoning, which we will make more accurate in the next section \S\ref{poinpol}, in this paragraph we aim to answer the following question.

\subsubsection{How many regular tessellations?}\label{howmany} We begin by seeking conditions for a regular polygon to tile the Euclidean plane. The first thing we may observe is that any vertex is always shared by at least three tiles. Moreover, the angles at any vertex sum up to $2\pi$. This simple fact immediately rules out all regular polygons with at least $7$ edges because the inner angles are all congruent and greater than $\frac{2\pi}{3}$: Notice that $3$ tiles with seven or more edges that share a vertex necessarily overlap. Only four polygons remain available as possible prototiles; namely the triangle, the square, the pentagon and the hexagon. Among these polygons, the only one to be excluded is the pentagon because its inner angles have magnitude $\frac{3\pi}{5}$. Three tiles with one vertex in common do not suffice to close up the figure; on the other hand four pentagons overlap. The remaining regular polygons, namely the triangle, the square and the hexagon have inner angles equal to $\frac{\pi}{3}$, $\frac{\pi}{2}$ and $\frac{2\pi}{3}$ respectively. We show that they can tile the whole Euclidean plane. Given a regular (equilateral) triangle $T$ in $\mathbb E^2$, we can place at any vertex of $T$ five regular triangles all isometric to $T$ itself. By repeating this infinitely many times, we cover the whole Euclidean plane. Notice that, since the triangles (tiles) are all isometric the tiling satisfy the edge-to-edge property and it is symmetric. We can also play the same game with the square and the regular hexagon and, in both cases, end up with a tiling of $\mathbb E^2$. The latter, in particular, provide the \textit{honeycomb} tessellation that we have already alluded to above. We can therefore affirm that there are three edge-to-edge regular tessellations of the Euclidean plane. By removing the edge-to-edge property, however, it turns out that there are seven more regular tessellations of $\mathbb E^2$, see \cite[Section \S4]{GS}. 

\subsubsection{Duality} There is a notion of duality for tessellations. Consider the tiling of $\mathbb E^2$ made by regular (equilateral) triangles. At any vertex, there are six triangular faces. By joining the centers of adjacent faces, \textit{i.e.} faces with one side in common, we obtain a regular tessellation of the Euclidean plane made entirely of regular hexagons. The opposite construction is also possible. We can start with the tessellation of $\mathbb E^2$ made by regular hexagon and then obtain a tessellation entirely made of regular triangles by joining the centers of adjacent hexagonal faces. Therefore we can go from a triangular to a hexagonal tessellation and vice versa: they are therefore \textit{dual}. The remaining tiling made of squares is \textit{self-dual}. In fact, by joining the centers of adjacent faces, a square tiling determines a new tiling whose prototile is a square isometric to any tile of the tessellation we have begun with. Duality plays an important role in the theory. In fact, dual tessellations have the same intrinsic nature despite the tessellations themselves may appear different. We shall describe a similar notion of duality for platonic solids in Section \S\ref{platsol}.

\subsection{Semi-regular tilings and beyond}\label{srt} In this paragraph we briefly discuss tilings with two or more prototiles. A tessellation is called \textit{semi-regular} if the prototiles are all regular polygons, possibly of different types, and any vertex has the same valence and configuration, see \cite[Vertex type]{DG}. Notice that regular tessellations are special cases of semi-regular tessellations. Semi-regular tilings of the Euclidean plane has been described by Kepler in his \textit{Harmonices Mundi}. He showed that there are exactly $11$ semi-regular tilings of $\mathbb E^2$; three of which are the regular tessellations described above in \S\ref{step}. The remaining $8$ tilings can be subdivided according to the number of prototiles:
\begin{itemize}
    \item six semi-regular tessellations have two prototiles: For five out of six of these tilings, a prototile is always a triangle that can be combined with a square prototile in two different ways, a hexagonal prototile in two different ways on a dodecagonal prototile. Finally, for the sixth tiling, the prototiles are a squares and an octagon.
    \item two semi-regular tessellations have three prototiles. The prototiles are a triangle, a square and a hexagon in one case and, for the second case, a square, a hexagon and a dodecagon.
\end{itemize}
\noindent It is worth mentioning that if we drop the assumption to have regular prototiles, then there are infinitely many tessellations of $\mathbb E^2$. For instance, any triangular or quadrilateral in itself can be a prototile for paving the Euclidean plane. We can even find pentagonal (non-regular) prototiles, \textit{e.g. Cairo pentagonal tiling}, named for its use as a paving design in Cairo (Egypt), see Figure \eqref{fig:cpt}.

\begin{figure}[!ht]
  \centering
  \includegraphics[scale=0.2,valign=t]{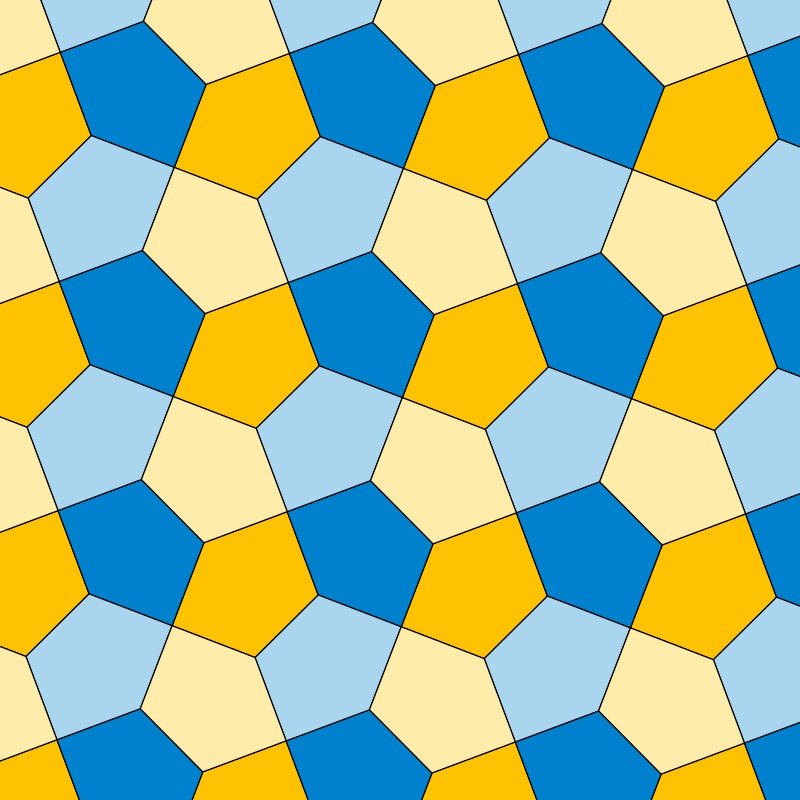}
  \caption{Cairo pentagonal tiling. Notice that this is not a regular tiling but it is a symmetric tiling of $\mathbb E^2$. \label{fig:cpt}}
\end{figure}

\smallskip

\noindent As already alluded to above, we shall introduce in Section \S\ref{platsol} Platonic solids. They will play an important role in the theory of tessellations because several regular tilings of the sphere $\mathbb S^2$ arise from them. In the same fashion, it is possible to show that semi-regular Platonic solids yield semi-regular tilings of $\mathbb S^2$, for an example see Figure \eqref{srhomb}. Finally, we shall see in \S\ref{tg} that any $n-$polygon can be a prototile for a regular tessellation of the hyperbolic plane. Although these have been known since many years, semi-regular tilings of the hyperbolic plane have been very recently studied by Datta-Gupta in \cite{DG}. Figure \eqref{rhomb} provides an example of semi-regular tiling of the hyperbolic plane.

\subsection{Going deeper} According to \cite{GS}, the mathematical theory of tessellations is pretty elementary in itself but rich of interesting facts and connections with other areas of mathematics. Regular and semi-regular tessellations in the Euclidean plane provide an excellent playground to begin with. We have introduced them above in a heuristic way, based on a logical reasoning, by observing which regular polygons can be used as prototiles. For spherical and hyperbolic tessellations the theory becomes more interesting as well as more challenging. In particular, it is no longer easy to find out which polygons can pave the sphere or the hyperbolic plane. Roughly speaking, this is due to the fact that both spherical and hyperbolic geometry are not locally isometric to the Euclidean geometry. Our heuristic approach in \S\ref{howmany} strongly relies on a classical fact: The sum of the inner angles of a Euclidean $n$-polygon is always equal to $(n-2)\pi$ and thence two regular $n$-polygons in $\mathbb E^2$ are always similar. These properties miserably fail for non-Euclidean geometries. Two regular $n$-polygons, both in spherical and hyperbolic geometry, may not be similar and the inner angles may sum up to different amounts. Therefore, for any $n$, there may be several, possibly infinitely many, pairwise not similar regular $n$-polygons. A priori, there may very well be a regular $n$-prototile for every $n$; that is a regular $n$-polygon that tiles the sphere or the hyperbolic plane. As we shall see in \S\ref{platsol} this will not be the case for the spherical tessellations but it will be for tilings of the hyperbolic plane.

\subsection{Structure of the paper} In Section \S\ref{poinpol}, we shall provide conditions for a polygon to be a prototile in the corresponding geometry. Section \S\ref{tg} is the central core of this Section. We shall introduce \textit{triangle groups} as the groups of symmetries of pavings made of triangular tiles, see for instance Figure \eqref{tg542}. As we shall see, any regular tessellation can be refined by dividing each tile into sub-tiles. By joining the center of any tile to its vertices and to the midpoint of any edge we obtain a new paving where the prototile is a triangle. It is an easy matter to check that the group of symmetries of any tessellation $\mathcal{T}$ is nothing but the group of symmetries of its triangular refinement. It is interesting to observe how different tessellation have the same group of symmetries. This is, for instance, the case of the triangular and hexagonal tessellation of the Euclidean plane described in \S\ref{howmany}. It is no coincidence that these two tilings are the dual of each other. In Section \S\ref{platsol}, we shall describe regular tessellation of the sphere $\mathbb S^2$. Interestingly, they almost all arise from Platonic solid. 

\subsection*{Acknowledgments}
\smallskip

\noindent I am grateful to Prakruti Kalsaria for careful reading and comments about the present work. \textit{Picture credits:} Most of the pictures has been completely borrowed from the web. Wikipedia has a very well-stocked catalog of hyperbolic plane tessellations corresponding to different triangle groups. All of these images are release into the public domain. Figures \eqref{rhomb}, \eqref{tg542} are attributed to by Parcly Taxel. Figure \eqref{fig:cpt} is attributed to David Eppstein. Figures \eqref{tg2453} and \eqref{tginf} are attributed to Tamfang. Figure \eqref{srhomb} is attributed to Tomruen.

\section{Poincar\'e's Theorem for compact polygons}\label{poinpol}

\noindent A surface is simply a topological manifold of dimension two. Let us start with the following general leading question that motivates the study of regular tessellations on surfaces.
\begin{quote}
    \textit{Given a polygon $\mathcal{P}$ and a surface $S$, does there exist a division of $S$ into non-overlapping congruent tiles each one congruent to $\mathcal{P}$?}
\end{quote}

\noindent In order to talk about notions as "polygons" and "congruence", one needs in the first place to endow a topological surface with a Riemannian metric $ds^2$ so that notions like, distance, angles and area make sense. We shall refer to these surfaces as \textit{geometric surfaces}. We begin with a short recall of geometry of surfaces. For an excellent account about geometry of surfaces the reader is invited to consult \cite{SJ}. 

\smallskip

\subsection{Geometry of simply connected surfaces}\label{ssec:geosimsur} Topologically, there are only two simply connected surfaces distinguished by compactness: The sphere which is compact and the plane which is not. In classical geometry the sphere is viewed as a figure in three-dimensional euclidean space analogous to the circle in the euclidean plane. The great interest for the sphere arises from the fact that it is not locally isometric to the plane. In particular, its intrinsic structure makes the sphere the first example of a non-Euclidean geometry. 

\smallskip

\begin{itemize}
    \item[] \textbf{Spherical geometry.} By adopting an extrinsic point of view, the $2-$dimensional sphere, usually denoted by $\mathbb S^2$, can be seen as the following locus 
\begin{equation}
    \mathbb S^2=\left\{\,(x,y,z)\in\mathbb R^3\,|\, x^2+y^2+z^2-1=0\,\right\}\subset \mathbb R^3.
\end{equation}

\noindent By using the polar coordinates $x=\sin\theta\cos\phi$, $y=\sin\theta\sin\phi$ and $z=\cos\theta$, the usual round metric on $\mathbb S^2$ can be explicitly written as $ds^2=d\theta^2\,+\,\sin^2\theta\, d\phi^2$. The group of isometries of $\mathbb S^2$, equipped with the round metric, is denoted by $\text{Iso}(\mathbb S^2)$, identifies with $\text{O}(3,\mathbb R)$. This extrinsic approach also allows to define the \textit{straight lines} for this geometry in a very simple way. Straight lines are given by great circles, \textit{i.e.} circles arising from the intersection of $\mathbb S^2$ with a $2-$dimensional plane passing through the origin in $\mathbb R^3$. 
\end{itemize}

\smallskip

\noindent A plane can be seen as a topological space homeomorphic to $\mathbb R^2$. Despite the sphere admits a unique conformal class of Riemannian metrics up to diffeomorphisms, see Killing-Hopf's Theorem below, a plane actually admits two non-equivalent conformal classes of Riemannian metrics up to diffeomorphisms. In our language, a plane admits two non-isometric geometries. 

\smallskip

\begin{itemize}
    \item[] \textbf{Euclidean geometry.} One of these is the well-known Euclidean geometry determined by the standard Riemannian metric $ds^2=dx^2+dy^2$. The Euclidean plane, denoted by $\mathbb E^2$, is defined as $\mathbb R^2$ equipped with the standard Euclidean metric and its group of isometries naturally identifies with $\text{Iso}(\mathbb E^2)\cong \textnormal{O}(2,\,\mathbb R) \ltimes \mathbb R^2$.
    
    \smallskip
    
    \item[] \textbf{Hyperbolic geometry.} For the other geometry, we restrict to $\mathbb H^2=\big\{(x,y)\in\mathbb R^2\,|\, y>0\big\}$ known as \textit{upper half-plane}. To find out an explicit homeomorphism between $\mathbb R^2$ and $\mathbb H^2$ is a routine exercise. The upper half-plane $\mathbb H^2$ equipped with the Riemannian metric $ds^2=y^{-2}(dx^2+dy^2)$ is a \textit{model} of the hyperbolic plane. We say "model" because there are several models of hyperbolic geometry in use and they are all equally entitled as hyperbolic plane. For this model, the group of isometries identifies with $\pglr$, with the action given by M\"obius transformation as follows
\begin{equation}
     \pglr\,\times\,\mathbb H^2\longrightarrow \mathbb H^2, \quad \begin{pmatrix} a & b\\ c & d \end{pmatrix},\, z\longmapsto \frac{az+b}{cz+d}.
\end{equation}

\noindent In this model, straight line for the hyperbolic metric are given by half-lines and half-circles orthogonal to $\partial\mathbb H^2=\{(x,y)\in\mathbb R^2\,|\, y=0\}$.

\smallskip

\noindent Another model which is worth to mention is the so-called \textit{Poincar\'e disk} $\mathbb D^2$. This is defined as $\mathbb D^2=\{\,z\in\C\,:\, |z|<1\,\}$ equipped with the Riemannian metric $ds^2=4\big( 1\,+\,|z|^2\big)^{-2}|dz|^2$. Of course, this model is equivalent to that of upper half-plane. Observe that $\mathbb H^2$ and $\mathbb D^2$ are both proper subset of $\rs$ -- the Riemann sphere -- and an explicit isometry between these two models is given by the mapping $j\in\pslc\cong\textnormal{Aut}(\rs)$ defined as
\[ j:\mathbb H^2\longrightarrow \mathbb D^2,\,\,\,j(z)=\frac{i\,z+1}{z-1}.
\]

\noindent For this second model, straight line for the hyperbolic metric are given by diameters and half-circles orthogonal to $\mathbb S^1=\partial \mathbb D^2$. It is straightforward to check that the isometry group of $\mathbb D^2$ is given by $j\,\pglr\, j^{-1}$. Throughout the present survey, we will use both models based on our convenience. 
\end{itemize}

\begin{rmk}
Historically speaking, the hyperbolic plane was the result of the search for a non-euclidean plane-a surface with unbounded straight lines and, for each line $\ell$ and point $P\notin\ell$, there are more than one line through $P$ which does not meet $\ell$.
\end{rmk}

\noindent Euclidean geometry, spherical geometry and hyperbolic geometry are distinguished by their \textit{scalar curvature}, a key notion in Riemannian geometry. For surfaces, the scalar curvature is exactly twice the Gaussian curvature which is an \textit{intrinsic} invariant of a surface endowed with a Riemannian metric. The sphere endowed with its round metric has scalar curvature $+1$, the Euclidean plane $\mathbb E^2$ has scalar curvature $0$ whereas any model of hyperbolic plane has scalar curvature $-1$. The so-called \textit{Theorema Egregium} by Gauss states that the scalar curvature is invariant under local isometries and hence we deduce that these three geometries are pair-wise non-isometric. 

\smallskip

\noindent In anyone of these three geometries, any pair of distinct points can be joined by a path of minimal length usually called segments or \textit{geodesics}. In these geometries, it turns out that any line segment can be continued indefinitely. This key property, not guaranteed by a generic Riemannian manifold, makes the Euclidean geometry, spherical geometry and hyperbolic geometry \textit{complete}. By Hopf-Rinow's Theorem, this is equivalent to say that any closed and bounded subset of $S$ is compact. An example of \textit{incomplete} Riemannian metric on a simply connected surface is the following: Consider the sphere as $\mathbb S^2=\mathbb{R}^2\,\cup\{\infty\}$, that is the one-point compactification of $\mathbb S^2$. The round metric on $\mathbb S^2$ restricts to a Riemannian metric on $\mathbb R^2$ which is now incomplete. This shows, in particular, that if we drop the completeness from our assumptions, then $\mathbb R^2$ carries other geometries. 

\smallskip

\noindent The following Theorem states that Euclidean geometry, spherical geometry and hyperbolic geometry are the only possible complete geometries in dimension two, with constant scalar curvature, up to possible rescaling of the metric.

\begin{khthm}\namedlabel{thm:khthm}{Killing-Hopf's Theorem}
A simply connected surface equipped with a complete Riemannian metric with constant scalar curvature is isometric to either the $2$-sphere with its round metric; the plane with its Euclidean metric; or the upper half-plane with its hyperbolic metric.
\end{khthm}

\smallskip

\noindent In what follows, with \textit{polygon} we shall mean a bounded region in $\mathbb E^2,\,\mathbb S^2$ or $\mathbb H^2$, along with its bounding circuit, that is described by a finite number of straight line segments connected to form a closed polygonal chain. In reference to tessellations, polygons are often called \textit{tiles} and we also adopt this convention. For positive integers $p,\,q$, with $q$-regular $p-$gon we shall mean a regular polygon with $p$ sides such that every inner angle has magnitude $\frac{2\pi}{q}$. We have already introduced the notion of tessellation in the introduction in full generality. In the next section we provide a group-theoretic incarnation of tessellations.

\subsection{From symmetric tessellations to groups} Let $S$ be anyone of $\mathbb E^2,\,\mathbb S^2$ or $\mathbb H^2$. Let $\mathcal T$ be a symmetric tessellation, that is a tessellation of $S$ such that for any pair of tiles $t$ and $\tau$ there is $g\in G=\text{Iso}(\,S)$ such that $g\cdot t=\tau$ and $g$ maps the whole $\mathcal{T}$ to itself. For a tessellation $\mathcal T$, we denote by $\text{Sym}(\mathcal T)$ its group of symmetries. In principle, given two tiles $t$ and $\tau$ there may be more than one symmetry that maps the first tile onto the second. Suppose $g,\,h\in\text{Iso}(\,S)$ are two isometries such that $g\cdot t=\tau$ and $h\cdot t=\tau$. Since each symmetry is of course a reversible transformation, in the sense that $t=g^{-1}\cdot \tau$, it follows that $g^{-1}h\cdot t=t$, that means $g^{-1}h$ fixes $t$. For any tile $t$ we denote by $\text{Stab}_{\,\mathcal T}(\,t\,)$ the \textit{stabiliser} of $t$; \textit{i.e.} the subgroup of $\text{Sym}(\mathcal T)$ of those symmetries that fix $t$. Geometrically speaking, $\text{Stab}_{\,\mathcal T}(\,t\,)$ can be regarded as the group of symmetries of $t$. Since $\mathcal{T}$ is assumed to be symmetric, any pair of tiles are congruent and hence they all have the same group of symmetries. On the other hand, as two tiles are located in different places on $S$, their respective stabilisers are not the same group but isomorphic.

\begin{lem}
Let $\mathcal T$ be a symmetric tessellation of a geometric surface $S$ and let $t$ and $\tau$ be two tiles. Then 
\begin{equation}
    \textnormal{Stab}_{\,\mathcal T}(\,t\,)\cong \textnormal{Stab}_{\,\mathcal T}(\,\tau\,).
\end{equation}
\end{lem}

\begin{proof}
Let $g\in\textnormal{Sym}(\mathcal T)$ such that $g\cdot t=\tau$. Let $h\in\textnormal{Stab}_{\mathcal T}(\,t\,)$, then $ghg^{-1}\in \textnormal{Stab}_{\mathcal T}(\,\tau\,)$. As a direct consequence it follows that $\textnormal{Stab}_{\,\mathcal T}(\,t\,)\cong \textnormal{Stab}_{\,\mathcal T}(\,\tau\,)$ as desired.
\end{proof}

\smallskip

\noindent Although $\textnormal{Sym}(\mathcal T)$ may be an infinite group, the stabiliser $\text{Stab}_{\,\mathcal T}(\,t\,)$ of any tile $t$ is always finite. Let $n\ge1$ be its cardinality -- observe the cardinality does not depend on $t$. There is a finite subdivision of $t$ into finitely many \textit{subtiles}, say $s_1,\dots,s_n$ such that for any $s_i$ and $s_j$ there exists a unique symmetry $g_{ij}\in\text{Stab}_{\,\mathcal T}(\,t\,)$ that relates them, that is $s_i=g_{ij}\cdot s_j$. For any other tile $\tau=h\cdot t$ there is a well-defined subdivision $\sigma_1,\dots,\sigma_n$, where $\sigma_i=h\cdot s_i$, such that for any $\sigma_i$ and $\sigma_j$ there exists a unique symmetry $h\,g_{ij}\,h^{-1}\in\text{Stab}_{\,\mathcal T}(\,\tau\,)$ that relates them. Therefore, a symmetric tessellation $\mathcal T$ always admits a well-defined subdivided tessellation into subtiles. Since the subtiles have been chosen within an original tile so that exactly one $g\in\textnormal{Sym}(\mathcal T)$ maps a given one of them onto another, the same holds for any two subtiles in the subdivided tessellation. 

\smallskip

\noindent Such a subdivided tessellation gives a better picture of the symmetry group $\textnormal{Sym}(\mathcal T)$. In fact, any tile of the subdivided tessellation may be regarded as a \textit{fundamental tile} -- this terminology is far from being casual as we shall see -- because it includes a representative from each $\textnormal{Sym}(\mathcal T)$-orbit, and each $\textnormal{Sym}(\mathcal T)$-orbit is represented at most once in the interior of the fundamental tile. Once a fundamental tile is fixed, say $t_o$, there is a natural identification between a geometric object like a tessellation of a geometric surface $S$ and a subgroup of its isometry $\textnormal{Iso}^+(\,S)$ as follows
\begin{equation}\label{eq:tesstosym}
    \mathcal T \,\longrightarrow \,\textnormal{Sym}(\mathcal T)
\end{equation}
that maps $t$ to the unique element $g\in\textnormal{Sym}(\mathcal T)$ such that $g\cdot t_o=t$.

\begin{ex}
The triangular tiling has the same symmetry group of the tiling made of regular hexagon. In fact: if we refine the tessellation of the Euclidean plane made of regular triangles by considering the barycentric subdivision of each face we will get a new regular tessellation, say $\mathcal{T}$, made of right triangles with angles $\frac{\pi}{2},\, \frac{\pi}{3}$ and $\frac{\pi}{6}$. This latter, on the other hand, is also obtained if we refine the tessellation made of regular hexagons. The group of symmetries of  $\mathcal{T}$ is denoted by $\Delta(2,3,6)$. It is not hard to show that any isometry $\Delta(2,3,6)$ preserving $\mathcal{T}$ also preserves the regular tessellation made of triangles as well as the regular tessellation made of hexagons.
\end{ex}

\begin{rmk}\label{rem:color}
For symmetric tessellations, a useful enhancement of these tessellations is obtained by coloring the subtitles alternately black and white. Pictorially, the whole tessellation looks like as a possibly infinite chessboard where the tiles are no longer square but a generic polygon. In each case, the coloring is emphasised on one copy of the original tile. The symmetries of the tessellation which preserve the coloring are precisely the orientation-preserving symmetries whereas the symmetries that switch the coloring correspond to orientation-reversing symmetries. 
\end{rmk} 

\smallskip

\noindent For a tessellation $\mathcal T$, we have determined a subdivided tessellation so that each pair of tiles of this latter one are related by a unique symmetry of $\textnormal{Sym}(\mathcal T)$. The subdivided tessellation is characterised by the fact that each tile has no longer new symmetries. 
If $\Gamma<\textnormal{Sym}(\mathcal T)$ is a subgroup of finite index $[\Gamma,\,\textnormal{Sym}(\mathcal T)]$ with cosets $\textnormal{Sym}(\mathcal T)\,\gamma_1,\dots,\textnormal{Sym}(\mathcal T)\,\gamma_n$ and if $t_o$ is a fundamental tile for $\textnormal{Sym}(\mathcal T)$, then $\gamma_1\cdot t_o\,\cup\,\cdots\,\cup\,\gamma_n\cdot t_o$ is a fundamental tile for $\Gamma$. The most important example in this sense is given by the following

\begin{prop}\label{prop:orpres}
Let $\mathcal T$ be a symmetric tessellation of $S$ and let $\textnormal{Sym}(\mathcal T)$ its group of symmetries with fundamental tile $t_o$. Suppose $\textnormal{Sym}(\mathcal T)$ includes an orientation-reversing element $\gamma$, then $t_o\cup \gamma\cdot t_o$ is a fundamental tile for the subgroup $\textnormal{Sym}^+(\mathcal T)$ of orientation-preserving symmetries. Moreover, $\gamma\in\textnormal{Sym}(\mathcal T)$ can be chosen so that $t_o\cup \gamma\cdot t_o$ is a polygon, that is a fundamental tile for $\textnormal{Sym}^+(\mathcal T)$.
\end{prop}

\noindent In the light of this Proposition, from now on we shall restrict our discussions to groups of orientation-preserving isometries.

\smallskip

\subsection{Fundamental polygon} In the present subsection we aim to determine whether a polygon $\mathcal P$ can be the fundamental tile for a symmetric tessellation $\mathcal T$. Throughout the present sub-section, let $S$ be anyone of the three classical geometries $\mathbb E^2,\,\mathbb S^2$ or $\mathbb H^2$. Let us begin with the following

\begin{defn}
Let $S$ be a geometric surface and let $\Gamma<\text{Iso}(\,S)$ be a group acting on $S$. We shall say that the action of $\Gamma$ is \textit{properly discontinuous} if for every compact set $K\subset S$ we have that $K\,\cap\,\gamma\cdot K=\phi$ for all but finitely many $\gamma\in\Gamma$.
\end{defn}

\begin{defn}[Fundamental domain] Let $\Gamma<\text{Iso}^+(\,S)$ be a group acting properly discontinuously on $S$. A closed subset $\Omega\subset S$ is called a \textit{fundamental domain} for $S$ if 
\begin{itemize}
    \item[1.] $\textnormal{int}(\Omega)$ is an open domain,
    \smallskip
    \item[2.] there is a fundamental set $F$ such that $\textnormal{int}(\Omega)\subset F\subset \Omega$, where a fundamental set for $\Gamma$ is set which contains exactly one point from each orbit in $S$.
\end{itemize}
\end{defn}

\noindent We shall say that $\Omega$ and its images via $\Gamma$ \textit{tessellates} $S$. In fact, if $\Omega$ is a fundamental domain for a group $\Gamma$ acting on $S$, then $\textnormal{int}(\Omega)\,\cap\,\gamma(\,\textnormal{int}(\Omega)\,)=\phi$ for all $\gamma\neq 1$ and 
\begin{equation}
    \bigcup_{\gamma\,\in\,\Gamma} \gamma(\Omega)= S.
\end{equation}

\noindent Based on these definitions, we can now state the following becomes a straightforward consequence.

\begin{prop}\label{prop:discact}
For a symmetric tessellation $\mathcal T$ of $S$, its group of symmetries $\textnormal{Sym}(\mathcal T)$ acts properly discontinuously on $S$.
\end{prop}

\noindent We now determine necessary and sufficient conditions for a polygon $\mathcal P$ to be a fundamental tile for a symmetric tessellation. In the light of Proposition \ref{prop:orpres} we shall restrict our discussion to groups of orientation preserving symmetries.

\subsubsection{Side and angle conditions} If a compact polygon $\mathcal P$ is a fundamental tile for a group $\Gamma$ of orientation-preserving isometries of $S$ then the following holds
\begin{itemize}
    \item[i.] for each side $s$ of $\mathcal P$ there is exactly one other side $s'$ of $\mathcal P$ of the form $s'=\gamma\cdot s$ with $\gamma\in\Gamma$ -- the elements $\gamma$ are called side-pairing transformations of $\mathcal P$, and
    \smallskip
    \item[ii.] if each side $s$ is identified with the corresponding $s'$, then each set of vertices identified as a result corresponds to a set of corners of $\mathcal P$ with angle sum $\frac{2\pi}{p}$ for some $p\in \mathbb Z^+$.
\end{itemize}

\noindent Let us show why these conditions holds, hence the reason why these are necessary conditions. We begin with by showing

\begin{proof}[Necessity of condition i.] Let $\gamma_1\cdot \mathcal P,\dots,\gamma_n\cdot \mathcal P$ be the tiles who share an edge with $\mathcal P$. Let $s_i$ be the edge shared by $\mathcal P$ and $\gamma_i\cdot \mathcal P$. Then, it is an easy matter to observe that $s_i'=\gamma_i^{-1}\cdot s_i$ is shared by $\gamma_i^{-1}\cdot\mathcal P$ and $\mathcal P$. Hence $s_i'$ is also a side of $\mathcal P$. It may very well happen that $s_i=s_i'$, in this case $\gamma_i=\gamma_i^{-1}$ that means $\gamma_i$ is a rotation of order $2$ because it is an orientation preserving isometry. Whenever this is the case, the midpoint of $s_i$ is fixed by $\gamma_i$ and it may be regarded as a new vertex of $\mathcal P$ and $s_i$ is split into two sub-segments and the mapping $\gamma_i$ swaps them. It remains to show that for any $s_i$ there is only one $s_i'$. If this is not the case, then there would be two points in the interior of $\mathcal P$ in the same $\Gamma$-orbit. However this contradicts the fact that $\mathcal P$ is a fundamental tile.
\end{proof}

\begin{proof}[Necessity of condition ii.] For a polygon $\mathcal P$, let $\{v_1,\dots,v_n\}$ be the set of its vertices. Two vertices $v_{i}$ and $v_{j}$ are in the same vertex cycle if and only if there is a mapping $\gamma\in\Gamma$ such that $\gamma(v_i)=v_j$. It is an easy matter to check that two cycles are disjoint or coincide. Let $\{v_{i_1},\dots,v_{i_k}\}$ be a vertex cycle and suppose $v_{i_1},\dots,v_{i_k}$ is the order of the corners induced by the side-pairing transformations. Then if $v=\gamma_{h}(v_{i_h})$ is an image of any such vertex under $\gamma_{h}\in\Gamma$, the corners of tiles which meet at $v$ are just images of the corners at $v_{i_1},\dots,v_{i_k}$ in the same cyclic order. The angle $2\pi$ at $v$ is an integer multiple of the angle sum of the corners at 
$v_{i_1},\dots,v_{i_k}$. The conclusion follows.
\end{proof}

\noindent We may expect that a polygon $\mathcal P$ that satisfies the side and angle conditions can always be used to tile the surface $\mathbb S^2$, $\mathbb E^2$ or $\mathbb H^2$ to which it belongs, and indeed this the statement of Poincar\'e's Theorem.

\subsubsection{Poincar\'e's Theorem} Our aim here is to show that the side and angle conditions are also sufficient for a compact polygon to be a fundamental polygon. Let $\mathcal P$ be a given polygon in $S$ equal to $\mathbb S^2$, $\mathbb E^2$ or $\mathbb H^2$ and let $\{\gamma_1,\dots,\gamma_n\}$ be a collection of isometries in $\text{Iso}^+(\,S)$ that realises the sides pairing $s_i\mapsto \gamma_i\cdot s_i$. We are going to prove that $\mathcal P$ is the fundamental region for the group $\Gamma=\langle \gamma_1,\dots,\gamma_n\rangle$. In other words, we are going to show that $\mathcal P$ tessellates $S$ and $\Gamma$ appears as the group of symmetries of such a tessellation.

\smallskip

\noindent For such a polygon $\mathcal P$, the gist of the idea is to take a collection of tiles $\{\,t_{\gamma}\,\}_{\gamma\,\in\,\Gamma}$, each one congruent to $\mathcal P$, labelled by the elements of $\Gamma$. Let $t_1$ be the (unique) tile labelled with $1\in\Gamma$. Such a tile will serve as the base tile for the following 

\smallskip

\begin{quote}
    \textbf{Construction:} \textit{For each $\gamma\in\Gamma$, let $t_{\gamma}$ be a copy of $t_1$. Notice that $\gamma:\,t_1\longrightarrow t_{\gamma}$ provides an isometry between $t_1$ and $t_{\gamma}$. We label the sides of $t_{\gamma}$ with the same labels $s_i$ as their preimages in $t_1$ via $\gamma$. Let $\Sigma$ be the surface formed from these tiles obtained by identifying the side $s_i'$ of $t_{\gamma}$ with $s_i$ of the tile $t_{\gamma\gamma_i}$, where $\gamma_i$ relates $s_i$ and $s_i'$ as $\gamma_i(s_i)=s_i'$.}
\end{quote}

\smallskip

\noindent Notice that this ensures, in particular, that $t_{\gamma\gamma_i}$ is adjacent to $t_{\gamma}$ as $t_{\gamma_i}$ is adjacent to $t_1$. Our aim now is to prove that $\Sigma\,=\,S$, where $S$ is equal $\mathbb S^2$, $\mathbb E^2$ or $\mathbb H^2$ to which $\mathcal P$ belongs. We begin with the following

\begin{prop}\label{prop:scomplete}
$\Sigma$ is a complete geometric surface.
\end{prop}

\begin{proof}
In the first place, one has to show that $\Sigma$ is a geometric surface. It is sufficient to show that $\Sigma$ is a surface as the global geometry is then determined by the local geometry of each single tile. It is immediate to notice that any point in the interior of $\gamma\cdot\mathcal P$ admits an open neighborhood homeomorphic to an open ball in $\mathbb R^2$. Moreover, since the sides of $\mathcal P$ are identified in pairs, interior points of any side also have disc neighborhoods. However, showing that vertices also have open neighborhoods homeomorphic to an open ball in $\mathbb R^2$ is less obvious and requires more detailed argument.

\smallskip

\noindent Let $\{v_1,\dots,v_k\}$ be a vertex cycle of $\mathcal P$. This means that the vertices are related by a cyclic sequence of side pairings $g_{i_1},\dots,g_{i_k}$. Let $\theta_i$ denote the magnitude of the vertex $v_i$. Since $\mathcal P$ is supposed to satisfy the angle condition, the sum $\theta_1+\cdots+\theta_k$ is equal to $2\pi/p$ for some $p\in\mathbb Z^+$. Let $r_i(v_i)$ denote the rotation of angle $\theta_i$ about the vertex $v_i$, in the sense from $s_{i_{j-1}}'$ to $s_{i_j}$. We have
\begin{equation}\label{eq:genform}
    g_{i_k}\,r_k(v_k)\cdot\cdots\cdot g_{i_2}\,r_2(v_2)\, g_{i_1}\,r_1(v_1)=\text{Id}
\end{equation}
because the isometry on the left-hand side is the identity on $s_{i_k}$. Moreover one may observe that 
\begin{equation}\label{eq:r2}
    r_2(v_2)=g_{i_1}\,r_2(v_1)\,g_{i_1}^{-1}
\end{equation}
because $g_{i_1}$ maps $v_1$ to $v_2$. By plugging \eqref{eq:r2} into \eqref{eq:genform}, we obtain
\begin{equation}\label{eq:newgenform}
    g_{i_k}\,r_k(v_k)\cdot\cdots\cdot g_{i_2}\, g_{i_1}\,r_2(v_1)\,r_1(v_1)=\text{Id}.
\end{equation}
In the same fashion, it is possible to observe that
\begin{equation}
    r_3(v_3)\,g_{i_2}\,g_{i_1}=g_{i_2}g_{i_1}r_3(v_1)
\end{equation}
and then, by plugging this latter into \eqref{eq:newgenform}, we obtain
\begin{equation}
    g_{i_k}\,r_k(v_k)\cdot\cdots\cdot g_{i_3}\,g_{i_2}\, g_{i_1}\,r_3(v_1)\,r_2(v_1)\,r_1(v_1)=\text{Id}.
\end{equation}
By iterating the same process $k$ times we eventually obtain
\begin{equation}
    g_{i_k}\cdot\cdots\cdot g_{i_1}\cdot r_k(v_1)\cdot\cdots\cdot r_1(v_1)=\text{Id}.
\end{equation}
Notice that $r_k(v_1)\cdot\cdots\cdot r_1(v_1)$ is a rotation of angle $\theta_1+\cdots+\theta_k=2\pi/p$ about $v_1$ and hence $g_{i_k}\cdot\cdots\cdot g_{i_1}$ is a rotation about the same vertex of angle $2\pi/p$. It follows that $\big(g_{i_k}\cdot\cdots\cdot g_{i_1}\big)^p=\text{Id}$. Therefore, the sequence of corners at a vertex of $\Sigma$ corresponding to the vertex cycle $\{v_1,\dots,v_k\}$ closes up after precisely $p$ repetitions of the cycle; namely when the angle sum is $2\pi$ as required.

\smallskip

\noindent It remains to show that $\Sigma$ is complete. It is an easy matter to show that any closed and bounded subset $K\subset \Sigma$ is contained in the union of finitely tiles, that is finite many copies of our polygon $\mathcal P$. Since such a union is compact and $K$ is closed we deduce that $K$ is also compact. Therefore $\Sigma$ is complete as desired.
\end{proof}

\noindent We are now ready to prove 

\begin{pthm}\namedlabel{thm:pthm}{Poincar\'e's Theorem}
A compact polygon $\mathcal{P}$ satisfying the side and angle conditions is a fundamental region for the group $\Gamma$ generated by the side-pairing transformations of $\mathcal{P}$.
\end{pthm}

\begin{proof}
Let $\mathcal P$ be a polygon that satisfies the side-angle conditions and realise $\Sigma$ as described above. As a consequence of Proposition \ref{prop:scomplete}, we already know that $\Sigma$ is complete. However we still do not whether it is simply connected. 

\smallskip

\noindent Let us assume in the first place that $\pi_1(\,\Sigma\,)=\{1\}$ -- we shall later prove that this is always the case. If this is the case, $\Sigma\,=\,S$ and the tessellation of $\Sigma$ by polygons $t_{\gamma}$ is a tessellation of the surface $S$; that is one of the surfaces $\mathbb R^2$, $\mathbb S^2$ or $\mathbb H^2$ containing the original polygon $\mathcal P$. We take $t_1$, the unique tile associated to $1\in\Gamma$, simply as $\mathcal P$. We aim to show that $\gamma\cdot\mathcal P\,=\, t_{\gamma}$. This can be proved by induction on $k$, where
\begin{equation}
    \gamma=\gamma_{i_1}^{\varepsilon_1}\cdot\,\cdots\,\cdot \gamma_{i_k}^{\varepsilon_k}
\end{equation}
and each $\varepsilon_i=\pm1$. The claim is true for $k=1$ because the neighbors $\gamma_i^{\pm1}\,\cdot\,\mathcal P$ of $\mathcal P$ are exactly $t_{\gamma_i^{\pm1}}$ by definition. Now we assume $\gamma'\cdot\mathcal P=t_{\gamma}$, where $\gamma$ is the product of $k-1$ generators or their inverses. By definition of $\Sigma$, we have that $t_{\gamma\,\gamma_i^{\pm1}}$ is adjacent to $t_{\gamma}$ exactly as $t_{\gamma_i^{\pm1}}$ is adjacent to $t_1$. On the other hand, by inductive hypothesis, $t_{\gamma}=\gamma\cdot\mathcal P$ and $t_{\gamma_i^{\pm1}}=\gamma_i^{\pm1}\cdot\mathcal P$. As a consequence we deduce
\begin{equation}
    t_{\gamma\,\gamma_i^{\pm1}}=\gamma\cdot t_{\gamma_i^{\pm1}}=\gamma\,\gamma_i^{\pm1}\cdot\mathcal P
\end{equation}
and the induction is complete. Therefore, if $\pi_1(\,\Sigma\,)=\{1\}$, the polygons $t_\gamma$ tessellating $\Sigma$ are precisely the polygons $\gamma\cdot\mathcal P$, for $\gamma\in\Gamma$. Therefore, $\mathcal P$ is a fundamental region for $\Gamma$ as desired.

\smallskip

\noindent We now drop the assumption $\pi_1(\,\Sigma\,)=\{1\}$. By classical covering theory the universal cover of $\Sigma$ is either $S=\mathbb E^2$ or $S=\mathbb H^2$. Consider the tessellation of $S$ obtained by lifting the tessellation of $\Sigma$. The gist of the idea now is to show that exactly one polygon of $S$ lies over each tile $t_{\gamma}\subset \Sigma$, thus proving that $\pi_1(\,\Sigma\,)=\{1\}$ and hence $\Sigma=S$. Let $t_1$ the fundamental tile for the tessellation on $\Sigma$ and let $\mathcal P$ and $\gamma\cdot \mathcal P$ two different lifts. For a point $p\in\mathcal P$, let $\gamma(p)\in\gamma\cdot\mathcal P$ and let $c:[0,1]\to S$ be a path joining them. In $S$, the path $c$ joins $\mathcal P$ from $\gamma\cdot\mathcal P$ by crossings finitely many edges corresponding to $\gamma=\gamma_{i_1}^{\varepsilon_1}\cdot\,\cdots\,\cdot \gamma_{i_k}^{\varepsilon_k}$. The path $c$ projects to a closed loop, say $\overline c$ in $\Sigma$ as $\mathcal P$ and $\gamma\cdot \mathcal P$ are two different lifts of $t_1$. We can notice that the sequence of crossings of $\overline c$ in $\Sigma$ coincides by construction of with the sequence of crossing of $c$ in $S$. Since $\overline c$ is a closed loop then $\gamma=1$ and hence there is only one tile over $t_1$ corresponding to $\mathcal P$. Since any tile can be the fundamental tile (tiles are indistinguishable in a symmetric tessellation) the conclusion holds for any tile. Therefore $\pi_1(\,\Sigma\,)=\{1\}$ as desired.
\end{proof}

\noindent As a direct consequence we obtain a symmetric tessellation $\mathcal T$ of $S$ having $\mathcal P$ as a prototile and the group $\Gamma$ turns out to be group of symmetries of $\mathcal T$. We now focus on regular tessellations as special and interesting examples of symmetric polygons.

\subsection{Regular tessellations}\label{ssec:regtess} A regular tessellation $\mathcal T$ of $S$ is a symmetric tessellation in which each tile is a regular polygon. Consider a two-dimensional tessellation with $q$ regular $p$-gons at each polygon vertex. Recall that for a $p-$gon in the Euclidean plane the sum of the inner angles is equal to $(p-2)\pi$. Since the angle at each vertex is $2\pi/q$, in the Euclidean plane
\begin{equation}
    \left(1-\frac{2}{p} \right)\pi=\frac{2\pi}{q}.
\end{equation}
This is equivalent to
\begin{equation}\label{eq:sse}
    \frac1p\,+\,\frac1q=\frac12.
\end{equation}
so
\begin{equation}
    (p-2)(q-2)=4.
\end{equation}
It is an easy matter to check that the only possible pairs $\{p,q\}$ are $\{3,6\}$, $\{4,4\}$ and $\{6,3\}$ corresponding to the triangle tessellation, the square tessellation and the hexagonal tessellation of $\mathbb E^2$. It is worth noticing that the symmetry of the pairs $\{3,6\}$ and $\{6,3\}$ reflects the fact that the triangle tessellation and the hexagonal tessellation are dual. On the other hand, the pair $\{4,4\}$ is symmetric and, in fact, the square tessellation is \textit{self-dual}. These pairs $\{p,\,q\}$ along with the present notation are known in literature as \textit{Schl\"afli symbols}.

\medskip

\noindent For a $q$ regular $p$-gon in the spherical plane $\mathbb S^2$ the sum of the inner angles is strictly greater than $(p-2)\pi$. This is consequence of the classical fact that, for a spherical triangle the sum of the inner angles is always greater than $\pi$. Therefore, for a regular $p$-gon in the spherical plane

\begin{equation}
    \left(1-\frac{2}{p} \right)\pi<\frac{2\pi}{q}.
\end{equation}
By reasoning as above, this equation is equivalent to
\begin{equation}\label{eq:sss}
    \frac1p\,+\,\frac1q>\frac12,
\end{equation}
Therefore
\begin{equation}
    (p-2)(q-2)<4.
\end{equation}
In this case there infinitely many solutions; in fact each pair $\{2,n\}$, for $n\ge 2$, satisfies \eqref{eq:sss}. A \textit{lune} on a sphere is a region bounded by two half great circles which meet at antipodal points. The word "lune" derives from "\textit{luna}", the Latin word for moon. For any $n\ge2$, an $n$-gonal \textit{hosohedron} is a tessellation of a sphere into $n$ lunes that all share the same two polar opposite vertices. An $n$-gonal hosohedron corresponds to the Schl\"afli symbol $\{2,n\}$. By symmetry, any pair $\{n,2\}$ is also a solution of \eqref{eq:sss} for any $n\ge2$. In this case we have \textit{dihedrons}, that is a polyhedron with two faces that share the same $n$ edges. As a spherical tiling, a regular dihedron is a polyhedron with two $n$-sided faces covering the sphere, each face being a hemisphere, and vertices on a great circle which are equally spaced. 

\smallskip

\noindent If we restrict to $p,\,q\ge3$ then there are finitely many solutions. In fact, there are only five solutions corresponding to the Schl\"afli symbols $\{3,3\}$, $\{3,4\}$, $\{4,3\}$, $\{3,5\}$ and $\{5,3\}$. There are some symmetric pairs which, once again, are far from being casual. The pairs $\{3,4\}$ and $\{4,3\}$, correspond to the tessellations determined by the cube and the octahedron respectively -- we shall see in Section \S\ref{platsol} how a Platonic solid determines a spherical tessellation. 
Similarly, the pairs $\{3,5\}$ and $\{5,3\}$ correspond to the tessellations determined by the icosahedron and the dodecahedron respectively.  The remaining pair $\{3,3\}$ is the Schl\"afli symbol of the spherical tiling arising from the tetrahedron. We shall consider the connection between spherical tiling and Platonic solids in Section \S\ref{platsol}.

\smallskip

\noindent We finally consider regular $q$ regular $p$-gons in the spherical plane $\mathbb H^2$ the sum of the inner angles is strictly less than $(p-2)\pi$. In this case, this is a consequence of the classical fact that, for a hyperbolic triangle the sum of the inner angles is always less than $\pi$. For a regular $p$-gon in the hyperbolic plane

\begin{equation}
    \left(1-\frac{2}{p} \right)\pi>\frac{2\pi}{q}.
\end{equation}
By reasoning as above, this equation is equivalent to
\begin{equation}\label{eq:ssh}
    \frac1p\,+\,\frac1q<\frac12,
\end{equation}
Therefore
\begin{equation}
    (p-2)(q-2)>4.
\end{equation}
\noindent Even in this case there are infinitely many solutions to \eqref{eq:ssh}. In fact, the following holds.

\begin{prop}
For any pair of positive integers $\{p,\,q\}$ that satisfies the inequality \eqref{eq:ssh} there exists a $q$-regular $p$-gon in $\mathbb H^2$. In particular, any such a polygon is the prototile for a regular tiling of $\mathbb H^2$. 
\end{prop}
\noindent Figure \eqref{tg542} depicts an example of tessellation of the hyperbolic plane in right-angles pentagons. A different perspective of the same tiling can be seen in the picture below.

\begin{figure}[!ht]
  \centering
  \includegraphics[scale=0.185]{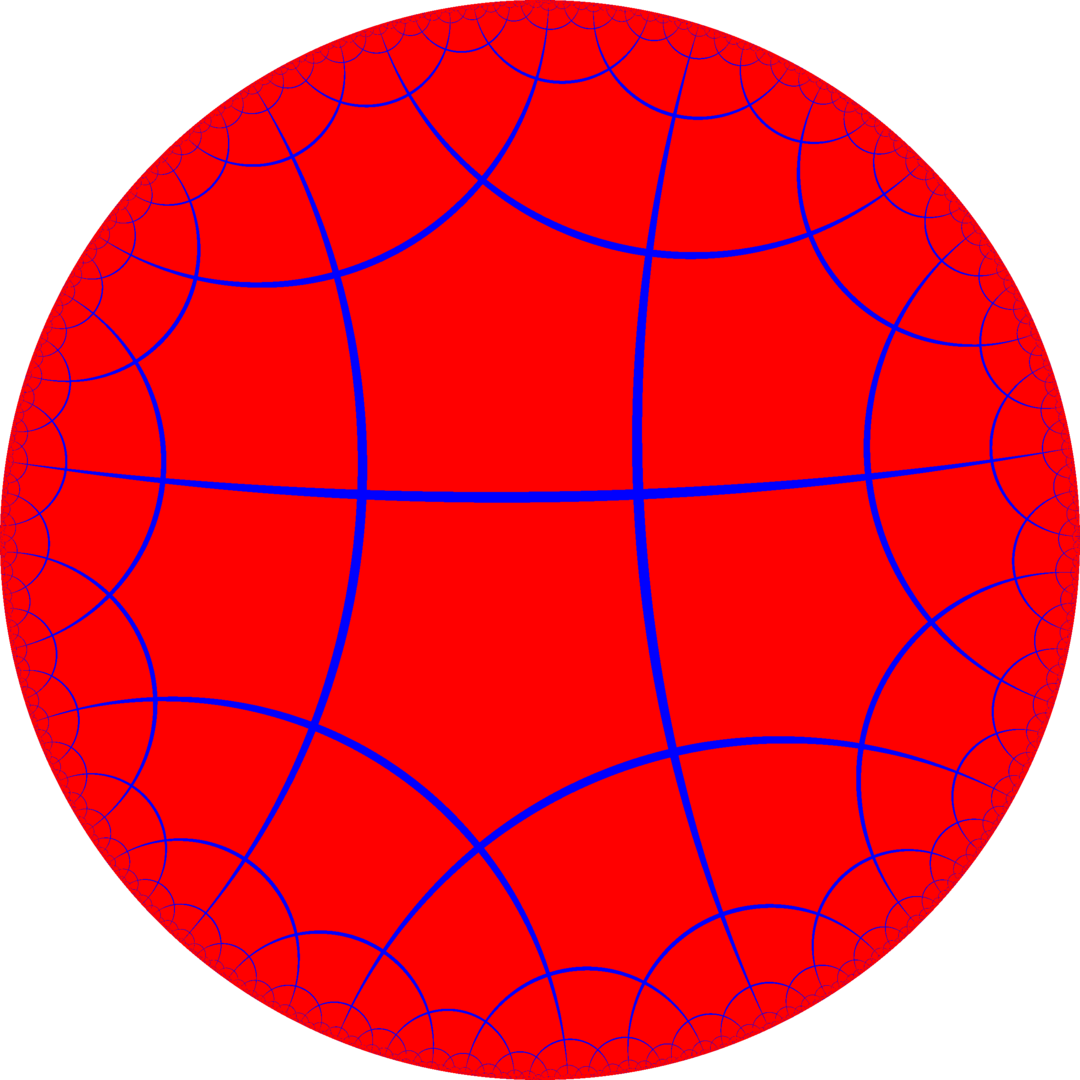}
  \caption{Triangular tessellation of the Poincar\'e disk in which each tile is a hyperbolic right-angles pentagon.}
  \label{tg2453}
\end{figure}

\begin{proof}
Let $p,\,q$ be a pair of integers that satisfy the inequality \eqref{eq:ssh}. In order to realise a $q$-regular hyperbolic $p$-gon it is better for us to work with the Poincar\'e disk. Consider $p$ rays $r_k$, for $k=1,\dots,p$, leaving from the origin $0\in\mathbb D^2$ and defined as
\begin{equation}
     r_k=\left\{ \,\rho\exp\left(\frac{2\pi\,i\,k}{p}\right)\,\,\Big|\,\,\rho\in\,[0,1)\, \right\}.
\end{equation}
\noindent For a fixed value $\rho\in[0,1)$ and define $v_k$ as the unique point on $r_k$ at distance $\rho$ from the origin. Let $\mathcal P_{\rho}$ be the polygon in $\mathbb D^2$ with vertices $v_1,\dots,v_p$. By construction, this is a regular hyperbolic $p$-gon. Let $\alpha(\rho)$ be the magnitude of any inner angle of $\mathcal P_{\rho}$ -- as it is regular they all have the same magnitude. It is an easy matter to check that the function $\rho\longmapsto \alpha(\rho)$ is continuous and that
\begin{equation}
    \lim_{\rho\to 1^-} \alpha(\rho)= 0 \quad \text{ and } \quad \lim_{\rho\to 0^+} \alpha(\rho)= \left( 1-\frac{2}{p}\right)\,\pi .
\end{equation}
In fact, let us write the hyperbolic area of $\mathcal P_{\rho}$ explicitly as a function of $\rho$
\begin{equation}
    \mathcal A_{\mathbb D^2}\left(\,\mathcal P_{\rho}\,\right)=(p-2)\pi\,-\,p\,\alpha(\rho).
\end{equation}
Notice that, as a function of $\rho$, it is continuous on the open interval $\big(\,0,\,1\,\big)$. For $\rho\mapsto 1^-$, at the limit $\mathcal P_{\rho}$ is an ideal regular polygon; that is all the vertices are at the infinity. In this case the hyperbolic area is as bigger as possible, that is $(p-2)\pi$ and the sum of the interior angles is equal to $0$. As a consequence, the first limit follows. For $\rho\to 0^+$, the hyperbolic area $\mathcal A_{\mathbb D^2}\left(\,\mathcal P_{\rho}\,\right)$ tends to $0$ and hence the second limit follows.

\smallskip

\noindent The mapping $\rho\longmapsto \alpha(\rho)$ is therefore monotone decreasing and for any value $\theta\,\in\,\left(\,0,\,\frac{p-2}{p}\,\pi\,\right)$ there exists a unique $\rho_o$ such that $\alpha(\rho_o)=\theta$. Since the pair $p,\,q$ satisfies the inequality \eqref{eq:ssh}, it follows that 
\begin{equation}
    \frac{2\pi}{q}<\left(1-\frac{2}{p}\right)\,\pi;
\end{equation}
in particular there exists $\rho_o$ such that $\alpha(\rho_o)=\frac{2\pi}{q}$. Therefore, $\mathcal P_{\rho_o}$ is the desired $q$ regular hyperbolic $p-$gon. Finally, notice that $\mathcal P_{\rho_o}$ satisfies the side and angle conditions and hence the second statement follows as a consequence of \ref{thm:pthm}.
\end{proof}

\smallskip

\section{Triangle tessellations and Triangle groups}\label{tg} 

\noindent Throughout the present section, let $S$ be anyone of $\mathbb E^2$, $\mathbb S^2$ or $\mathbb H^2$. A \textit{triangle tessellation} of $S$ is a symmetric tessellation $\mathcal T$ whose prototile is a triangle $T$. A \textit{triangle group} is a group $\Delta < \textnormal{Iso}(\,S)$ generated by reflections across the sides of the triangle $T$. As we shall see, each triangle group appears as the symmetry group $\textnormal{Sym}(\mathcal T)$ of a tessellation $\mathcal T$ of the Euclidean plane, the sphere, or the hyperbolic plane made of triangular tiles. Each of these triangle is a then a prototile, namely a fundamental domain for the action of $\textnormal{Sym}(\mathcal T)$. Good sources for triangle groups are Magnus \cite{MW} and Ratcliffe \cite{RJ}. 

\subsection{Triangular tiles}\label{ssec:tritiles} The simplest polygon that satisfies the side and angle conditions are quadrilaterals obtained by \textit{doubling} triangles. More precisely, for a triangle $T$, let $T'$ be triangle obtained by reflecting $T$ along one of its edges. By construction, $T$ and $T'$ have an edge in common; therefore the resulting shape $T\,\cup\, T'$ is a quadrilateral $Q$. We shall say that $Q$ is obtained by doubling $T$. Of course, $Q$ highly depends on which side we reflect $T$. 

\smallskip

\begin{defn}
A triangle $T$ has \textit{rational angles} if the interior angles have the form $\pi/p$, $\pi/q$ and $\pi/r$ for some positive integers $p,q,r\in\mathbb Z_{\ge2}$. It is customary to allow $\infty$ as a possible value of $p,q,r$, possibly all. A triangle $T$ is called $(p,q,r)$-triangle if it has rational angles $\pi/p$, $\pi/q$ and $\pi/r$. 
\end{defn} 

\begin{rmk}
Whenever at least one among these integers is $\infty$, the triangle $T$ is no longer a compact subset of $S$ and the corresponding vertex lies at "the infinity". A vertex at the infinity is often called \textit{ideal vertex} and we shall make the convention that $\frac{\pi}{\infty}=0$, that is every ideal vertex has magnitude $0$. Notice that this is never the case in spherical geometry because, being compact, any closed subset of $\mathbb S^2$ is compact.
\end{rmk}

\noindent A symmetric tessellation of $S$ having a quadrilateral as a prototile comes from a tessellation, say $\mathcal T$, generated by reflections in the sides of the triangle. More precisely, a quadrilateral tessellation can be seen as a tiling corresponding to an index two subgroup of $\textnormal{Sym}(\mathcal T)$, see Proposition \ref{prop:orpres}. The question of existence is settled by arguments which depend on the value of the angles $\pi/p$, $\pi/q$ and $\pi/r$. It is possible to show that, if a triangle $T$ is rational, a quadrilateral $Q$ obtained by doubling $T$ satisfies the side-angle conditions. The existence of a symmetric tessellation with $Q$ as a prototile is guaranteed by \ref{thm:pthm}. It is a classical fact in Euclidean geometry that the sum of inner angles of any triangle equals $\pi$. This means that an Euclidean triangle with rational angles $\pi/p$, $\pi/q$ and $\pi/r$ satisfies the equation
\begin{equation}\label{eq:innang}
    \frac{\pi}{p}+\frac{\pi}{q}+\frac{\pi}{r}=\pi \qquad \textnormal{ that is }\qquad \frac1p+\frac1q+\frac1r=1.
\end{equation}

\noindent On the other hand, recall once again that the sum of the inner angles of any spherical triangle is greater than $\pi$ and the sum of the inner angles of any hyperbolic triangle is lower than $\pi$. We have therefore similar relations as in \eqref{eq:innang}
\begin{align}
\frac{\pi}{p}+\frac{\pi}{q}+\frac{\pi}{r}>\pi \,\,\,\,\,\Longleftrightarrow\,\,\,\,\, \frac1p+\frac1q+\frac1r>1\\
\text{ } \notag\\
\frac{\pi}{p}+\frac{\pi}{q}+\frac{\pi}{r}<\pi \,\,\,\,\,\Longleftrightarrow\,\,\,\,\, \frac1p+\frac1q+\frac1r<1
\end{align}

\noindent The reader may notice that we have already seen a similar equation in \eqref{eq:sse}, \eqref{eq:sss} and \eqref{eq:ssh} in the very special case of $r=2$. Clearly, this is far from being a coincidence. In fact, a $q$-regular $p$-gon $\mathcal P$ can be subdivided into $2p$ triangles with inner angles $\pi/p$, $\pi/q$ and $\pi/2$. In other words, any regular tessellation seen above admits a sub-tessellation into right-angle triangles. Let us define the quantity
\begin{equation}
    \chi(p,q,r)= \frac1p+\frac1q+\frac1r-1\,\,.
\end{equation}
A triple $(p,q,r)$ is said to be \textit{spherical, Euclidean or hyperbolic} according as the quantity $\chi(p,q,r)$ is \textit{positive, zero or negative}. It is not hard to check that the spherical triples are $(2,3,3)$, $(2,3,4)$, $(2,3,5)$ and $(2,2,n)$ for $n\ge2$. In Section \S\ref{platsol} we shall see the relationship between these triples and the Schl\"afli symbols corresponding to spherical tessellations. The Euclidean triples are $(3,3,3)$, $(2,3,6)$, $(2,4,4)$ and a special one given by $(2,2,\infty)$. This latter triple corresponds to a tiling of $\mathbb E^2$ into infinite half-strips. These half-strips are non-compact triangles with two right-angles and one ideal vertex. The other triples, instead, correspond to tilings of $\mathbb E^2$ by compact triangles. For instance, the $(3,3,3)$ tiling results from a triangulation of $\mathbb E^2$ into equilateral triangles. The $(2,3,6)$ tiling also result from a triangulation of $\mathbb E^2$ into equilateral triangles each one divided into $6$ right-angle triangles. The same tiling arises by dividing e regular hexagon into $12$ right-angle triangles. Finally, $(2,4,4)$ arises from the square tessellation divided into $8$ isosceles right-angle triangles. All the other triples are hyperbolic. 

\smallskip

\begin{figure}[!ht]
  \centering
  \includegraphics[scale=0.185]{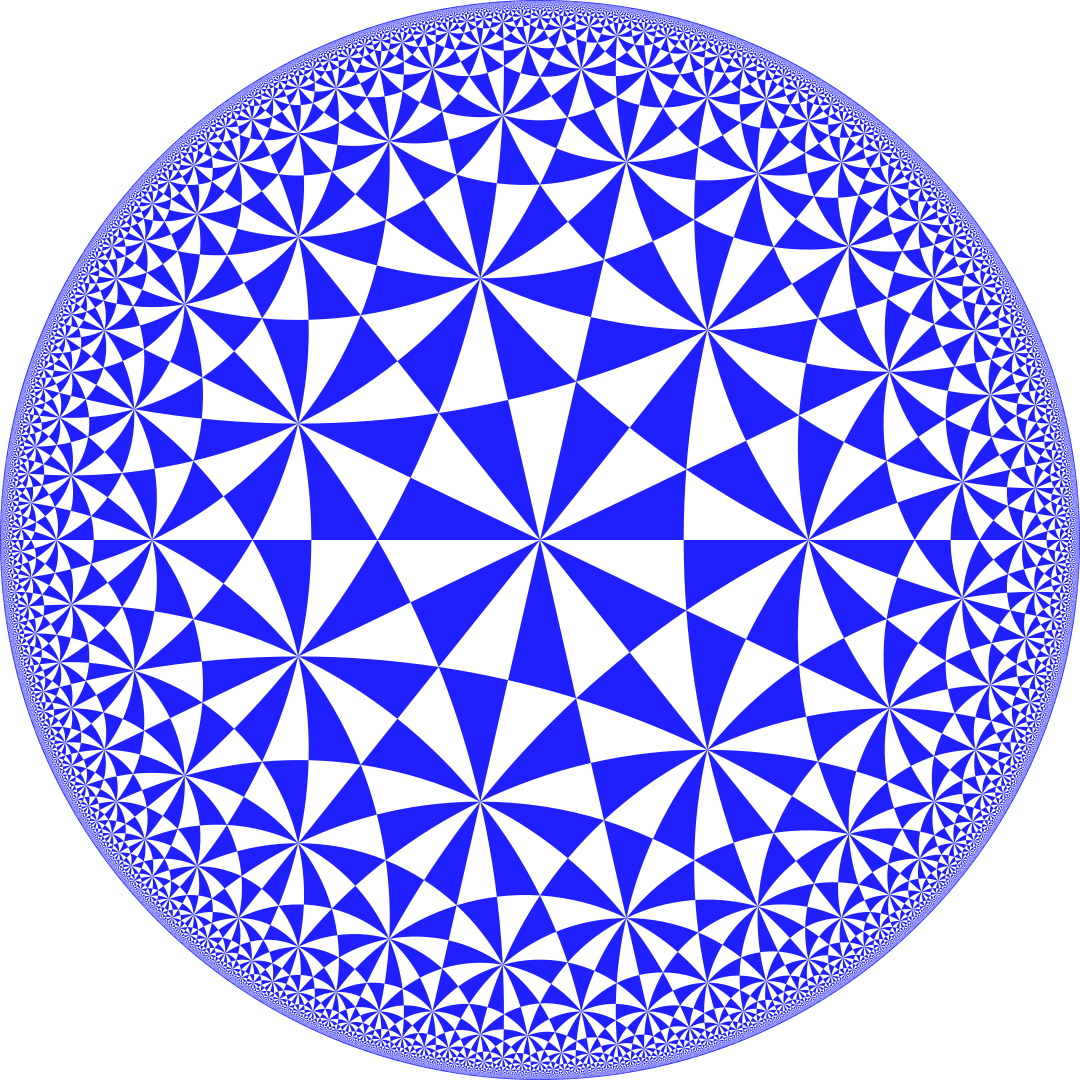}
  \caption{Triangular tessellation of the Poincar\'e disk where each tile is isometric to a hyperbolic triangle having inner angles of magnitude $\Big(\frac{\pi}{2},\,\frac{\pi}{3},\,\frac{\pi}{7}\Big)$}
  \label{tg237}
\end{figure}

\begin{ex}
The most important triple is surely the hyperbolic triple $(2,3,7)$. In $\mathbb D^2$, it is possible to place $336$ copies of the $(2,3,7)$-triangle in order to create a $7$-regular $14$-gon; that is a regular $14$-gon such that the magnitude of each inner angle is $\frac{2\pi}{7}$. According to Remark \ref{rem:color} above, we may coloring the triangles alternatively black and white. Two adjacent tiles with different colors provide a quadrilateral obtained by doubling the $(2,3,7)$ triangle. We can recognise $168$ such quadrilaterals that tile our $14$-gon. Label any side with $1$ and then label the others accordingly in positive cyclic order. If the side $2i + 1$ of such a polygon is identified with side $2i + 6$ $(\,\text{mod } 14\,)$, the result turns out to be a genus $3$ surface with two vertices. These two vertices correspond to the sets of even and odd vertices of our $14$-gon; it is easy to check that they both have angle sum $2\pi$. Thus, we have a genus $3$ surface equipped with a hyperbolic metric, that is a Riemannian metric with constant Gaussian curvature equal to $-1$. This surface provides the very first example, in several sense, of the so-called \textit{quasi-platonic} surfaces. We shall back on quasi-platonic surfaces in subsection \S\ref{ssec:qpsurf}.
\end{ex}


\subsection{Triangle groups}\label{tg2} Let $(p,q,r)$ be any triple as above. We may associate to any such a triple a group, for the moment denoted by $W(p,q,r)$, generated by elements $-1,\,\delta_p,\,\delta_q,\,\delta_r$, with $\,-1$ being central in $W(p,q,r)$, and subject to the relations $(-1)^2=1$ and 
\begin{equation}
    \delta_p^p=\delta_q^q=\delta_r^r=\delta_p\delta_q\delta_r=-1,
\end{equation}
where, by convention, we set $\delta_{\infty}^{\infty}=-1$. A $(p,q,r)-$\textit{triangle group}, denoted by $\Delta(p,q,r)$, is defined as the center-free quotient $W(p,q,r)/\{\pm1\}$.

\begin{rmk}
It is worth mentioning that we may order the generators so that $p\le q\le r$ without loss of generality. In fact, one may observe that $\delta_p\delta_q\delta_r=-1$ is invariant under cyclic permutations so 
\begin{equation}
    W(p,q,r)\,\cong\, W(q,r,p)\,\cong\,W(r,p,q).
\end{equation} 
Moreover, a mapping that sends each generator to its inverse yields an isomorphism $W(p,q,r)\,\cong\,W(r,q,p)$. It is an easy matter that these isomorphisms all boil down to the quotient by factoring $\{\pm1\}\cong\Z_2$.
\end{rmk}

\noindent The name "triangle group" is now strongly motivated by the fact that these groups appear as the group of symmetries of some triangular tiling of $\mathbb E^2,\,\mathbb S^2$ or $\mathbb D^2$. More precisely, $W(p,q,r)$ appears as the full group of symmetries of the triangular tessellation $\mathcal T$ of $S$ having a $(p,q,r)-$triangle, say $T$, as the prototile. This is a consequence of \ref{thm:pthm}. Of course, $S$ is $\mathbb E^2,\,\mathbb S^2$ or $\mathbb D^2$ depending on whether the triple is Euclidean, spherical or hyperbolic. As a consequence, its quotient $\Delta(p,q,r)$ appears as the the full group of symmetries of the symmetric tiling of $S$ having the quadrilateral obtained by doubling $T$ as a prototile. We may observe that $\Delta(p,q,r)$ is an index two subgroup of $W(p,q,r)$ and it comprises all and only orientation-preserving symmetries of $\mathcal T$, compare with Proposition \ref{prop:orpres}. In principle, one might be tempted to define a triangle group as the group $W(p,q,r)$ rather than its quotient $\Delta(p,q,r)$. However, as it is more convenient to work with groups of symmetries that preserve the orientation, it is customary to define its center-free quotient $\Delta(p,q,r)$ as triangle group.

\medskip

\noindent Let us provide a few examples for the reader convenience. In fact, it turns out that some triangle groups are isomorphic to other well-known groups. Generally, one may notice that, if $p,q\ge2$ and $r=\infty$, then there is a canonical isomorphism $\Delta(p,q,\infty)\,\cong\,\mathbb Z_p\,*\,\mathbb Z_q$. This simple observation leads to some interesting examples as follows.

\begin{figure}[ht!] 
\centering
\begin{tikzpicture}[scale=0.95, every node/.style={scale=0.75}]
\definecolor{pallido}{RGB}{221,227,227}

\draw[black, thick] (-5,0)--(5,0);
\draw[black, thin, dashed] (-4,-3)--(-4,3);
\draw[black, thin, dashed] (-2,-3)--(-2,3);
\draw[black, thin, dashed] (0,-3)--(0,3);
\draw[black, thin, dashed] (2,-3)--(2,3);
\draw[black, thin, dashed] (4,-3)--(4,3);

\pattern [pattern=north east lines, pattern color=pallido] (-4,0)--(-2,0)--(-2,3)--(-4,3)--(-4,0);
\pattern [pattern=north east lines, pattern color=pallido] (0,0)--(2,0)--(2,3)--(0,3)--(0,0);
\pattern [pattern=north east lines, pattern color=pallido] (4,0)--(5,0)--(5,3)--(4,3)--(4,0);
\pattern [pattern=north east lines, pattern color=pallido] (0,0)--(-2,0)--(-2,-3)--(0,-3)--(0,0);
\pattern [pattern=north east lines, pattern color=pallido] (4,0)--(2,0)--(2,-3)--(4,-3)--(4,0);
\pattern [pattern=north east lines, pattern color=pallido] (-4,0)--(-5,0)--(-5,-3)--(-4,-3)--(-4,0);

\node at (-0.25,0.25) {$O$};
\node at (2.25,0.25) {$O'$};

\fill [black] (2,0) circle (1.5pt);
\fill [black] (0,0) circle (1.5pt);

\draw[black, thin, ->] (-1.85,1.5)--(1.85,1.5);

\node at (-1,1.75) {$r_O\,r_{O'}$};
    
\end{tikzpicture}
\caption{$\Delta(2,2,\infty)$ is generated by two rotations, say $r_O$ and $r_{O'}$ of order $2$ about two distinct points $O$ and $O'$. It is an easy matter to check that $r_O\,r_{O'}$ is a translation of translation length equal to $2d_{\mathbb E^2}(O,\,O')$. Any shadow region is a fundamental domain for $W(2,2,\infty)$ whereas any infinite strip is a fundamental domain for $\Delta(2,2,\infty)$.}
\label{fig:tri22inf}
\end{figure}
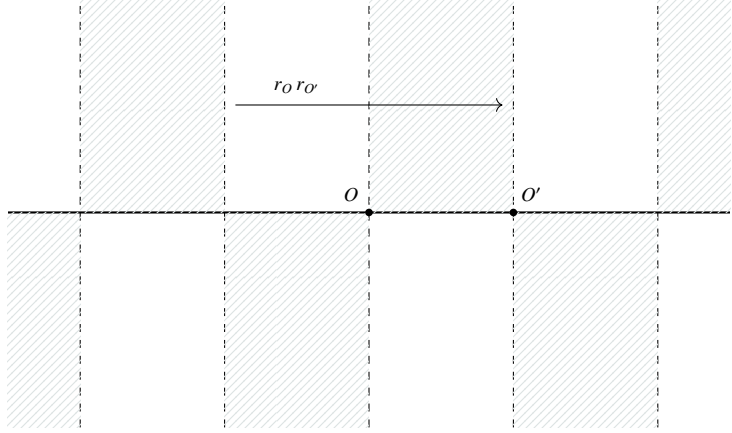

\begin{ex}
In the very particular case $p=q=2$, the triple $(2,2,\infty)$ is Euclidean and $\Delta(2,2,\infty)\cong\Z_2*\Z_2$ can be geometrically realised as the group of isometries of $\mathbb E^2$ generated by two rotations about two distinct points. The element of infinite order correspond to a translation along the unique straight line passing through these points with translation length equal to twice of their distance. See Figure \eqref{fig:tri22inf}.
\end{ex}

\begin{ex}
For the spherical triples $(2,3,3)$, $(2,3,4)$ and $(2,3,5)$ we have the following isomorphisms
\begin{equation}
    \Delta(2,3,3)\cong A_4,\qquad \Delta(2,3,4)\cong \mathfrak S_4, \qquad \Delta(2,3,5)\cong A_5.
\end{equation}
As we shall see, they correspond to orientation-preserving symmetry groups of Platonic solids, see \S\ref{platsol}.
\end{ex}

\begin{ex}
Another example which is worth of mentioning is that of $\Delta(2,3,\infty)\cong\Z_2*\Z_3\cong\text{PSL}(2,\Z)$, the modular group. More precisely, $\text{PSL}(2,\Z)$ is well-known to be generated by the matrices 
\begin{equation}
    S=
    \begin{pmatrix}
    0 & -1\\
    1 &  0
    \end{pmatrix}\,\,\,\,
    \text{ and }\,\,\,\,
    T=
    \begin{pmatrix}
    1 & 1\\
    0 &  1
    \end{pmatrix}.
\end{equation}
\noindent It is straightforward to check that $S^2=1$ whereas $T$ has infinite order. Moreover, their product $U=ST$ is a transformation of order $3$. Therefore, $\text{PSL}(2,\Z)$ admits a finite presentation
\begin{equation}
    \text{PSL}(2,\Z)=\langle \,S,\,U\,\,\,|\,\,\, S^2=I,\,\,U^3=I \,\rangle \cong \Z_2\,*\,\Z_3. 
\end{equation}
The mapping $\phi$ that associates $S\longmapsto\delta_2$, $U\longmapsto\delta_3$ and $T\longmapsto\delta_{\infty}$ naturally extends to the desired isomorphism $\phi:\text{PSL}(2,\Z)\longrightarrow \Delta(2,3,\infty)$. Let us consider the group $W(2,3,\infty)$ for a moment. A fundamental domain for this latter is a triangle $T\subset \mathbb H^2$ with vertices at $\zeta_4=i$, $\zeta_3=e^{\frac{2\pi\,i}{3}}$ and at $\infty$. We double $T$ by taking the triangle $T'$ as the mirror image of $T$ in the imaginary axis. Then region $F$ defined as
\begin{equation}
    F=T\cup T'=\left\{ z\in\mathbb H\,\,\big|\,\, -\frac12\le \Re(z)\le \frac12,\,\,|z|\ge1 \right\}
\end{equation}
turns out to be a fundamental region for $\Delta(2,3,\infty)$. The quotient space $\mathcal M_1=\mathbb H\,/\,\Delta(2,3,\infty)$ is of remarkable importance and it corresponds to the moduli space of Elliptic curves, see \S\ref{ssec:ellcur}.
\end{ex}

\smallskip

\begin{ex}
For $q\ge3$, the \textit{Hecke group}, denoted by $H_q$, is defined as the subgroup of $\pslr$ generated by
\begin{equation}
    S:\,z\longmapsto -\frac1z \,\,\text{ and }\,\, U_q:\,z\longmapsto z+2\cos\frac{\pi}{q}.
\end{equation}
For $q=3$, the group $H_3$ is isomorphic to $\text{PSL}(2,\Z)$ and hence isomorphic to $\Delta(2,3,\infty)$. More generally, one has $H_q\cong \Z_2\,*\,\Z_q$ and hence, by the same argument above, $H_q\cong \Delta(2,q,\infty)$.
\end{ex}

\smallskip

\begin{ex}
Another group which is worth of interest is $\Delta(\infty,\infty,\infty)\cong \Z*\Z$. It is a classical fact in hyperbolic geometry that two triangles are isometric as soon as they have same inner angles. In particular two ideal triangles, \textit{i.e.} a triangle with all the vertices at the infinity, are always isometric. As a direct consequence, we can deduce that any ideal triangle in $\mathbb H^2$ is a fundamental domain for $\Delta(\infty, \infty, \infty)$. See Figure \eqref{tginf}. This triangle group correspond to the principal congruence subgroup $\Gamma(2)<\text{PSL}(2,\mathbb Z)$ defined as
\begin{equation}
    \Gamma(2)=\left\{ \,\,g\in\text{PSL}(2,\mathbb Z)\,\,\big|\,\, g\equiv \pm\,\text{Id}\,\,(\text{mod}\,2)\,\,\right\}.
\end{equation}
The reader may observe that $\Gamma(2)$ is a normal subgroup of index six in the modular group $\text{PSL}(2,\mathbb Z)$ with the quotient group being $\text{PSL}(2,\mathbb Z)\,/\,\Gamma(2)\cong\mathfrak{S}_3$.
\begin{figure}[!ht]
  \centering
  \includegraphics[scale=0.325]{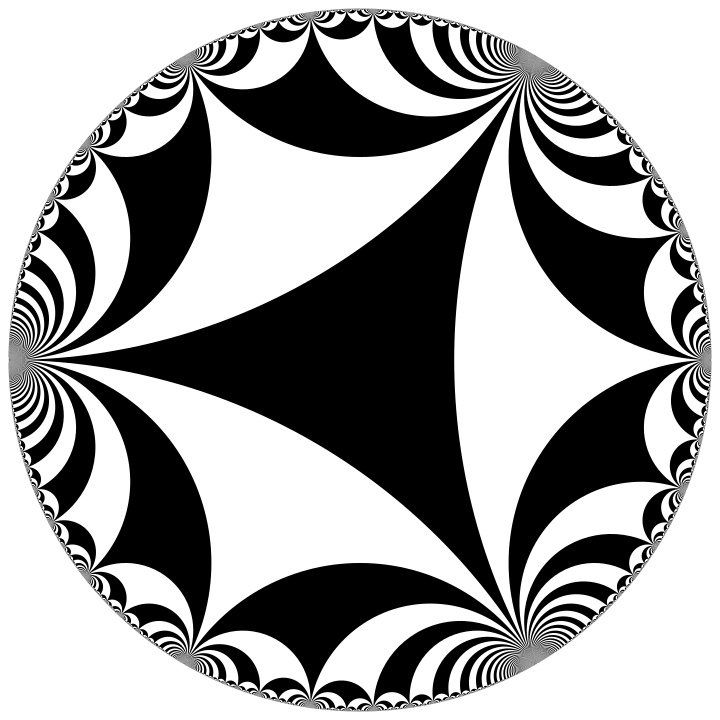}
  \caption{Triangular tessellation of the Poincar\'e disk where each tile is isometric to an ideal triangle. The corresponding triangle group is $\Delta(\infty, \infty, \infty)$.}
  \label{tginf}
\end{figure}
\end{ex}

\smallskip

\noindent We finally focus on triangle groups $\Delta(p,q,r)$ of hyperbolic type, \textit{i.e.} $\frac1p+\frac1q+\frac1r<1$. We have the following

\begin{thm}\label{thm:tribihospher}
Let $\Delta(p,q,r)$ be a triangle group of hyperbolic type. Then $\Delta(p,q,r)$ acts properly discontinuously on $\mathbb H^2$ and the quotient $\mathbb H^2/\Delta(p,q,r)$ is biholomorphic to the Riemann sphere $\rs$. More precisely, there is a meromorphic function $j:\mathbb H^2\to \rs$ which is $\Delta(p,q,r)$-invariant, namely it satisfies $j\big(g(z)\big)=j(z)$ for all $z\in\mathbb H^2$ and all $g\in\Delta(p,q,r)$. If $T$ is a fundamental domain for $W(p,q,r)$ and $F=T\cup T'$ is a fundamental domain for $\Delta(p,q,r)$, then $j$ maps the interior of $T$ and $T'$ biholomorphically onto $\mathbb H^2$ and $-\mathbb H^2$ and the edges onto $\mathbb R\mathbf{P}^1$ and the vertices onto $0,1,\infty$ with multiplicities $p,q,r$.
\end{thm}

\begin{proof}
We begin by observing that $\Delta(p,q,r)$ acts on $\mathbb H^2$ by M\"obius transformations. A triangle $T\subset \mathbb H^2$ with inner angles $\pi/r$, $\pi/q$ and $\pi/r$ satisfies the side angle conditions and hence, by \ref{thm:pthm}, it turns out to be a fundamental domain for a symmetric tessellation and the group $W(p,q,r)$ generated by side-pairing transformations is its group of symmetries. Moreover, the action of $W(p,q,r)$ is properly discontinuous by Proposition \ref{prop:discact}. By doubling $T$, as in subsection \S\ref{ssec:tritiles}, we obtain a quadrilateral which is a fundamental domain for the action of the index two subgroup $\Delta(p,q,r)$. As the action of $W(p,q,r)$ is properly discontinuous the same holds for any proper subgroup.

\smallskip

\noindent We next define the mapping $j:\mathbb H^2\longrightarrow \rs$ as follows. We invoke the Riemann mapping theorem to claim the existence of a mapping $j$ from the interior of $T$ to $\mathbb H^2$. A much stronger version asserts that $j$ extends continuously on the boundary and normalised so that the vertices of $T$ are mapped to $0,1,\infty$. Then $j$ extends to $T'$ as a consequence of Schwarz's reflection principle. It is possible to iterate this process by applying the same idea on all boundary lines of all images $g(F)$ where $F=T\,\cup\,T'$ and $g\in\Delta(p,q,r)$. We then obtain a well-defined holomorphic function on $\mathbb H$ at least outside the $\Delta$-orbits of the vertices of $F$. For these points, we finally apply Riemann's removable singularity. We finally notice that the mapping $j:\mathbb H^2\longrightarrow \rs$ is $\Delta(p,q,r)$-invariant by construction.
\end{proof}

\subsection{Inclusions of triangle groups}

\noindent In this subsection we still consider triangle groups of hyperbolic type. A \textit{Fuchsian} group $\Gamma$ is a discrete subgroup of $\pslr$. By definition, any triangle group is a Fuchsian group however the opposite is far from being true. There are, in fact, plenty Fuchsian group which are not triangle groups. A triangle group $\Delta(p,q,r)$ is \textit{maximal} if it cannot be properly embedded in any Fuchsian group. A result by Singerman \cite{SD} and Greenberg \cite[Theorem 3B]{GL}, any Fuchsian group containing a triangle group is itself a triangle group. Singerman in \cite{SD} lists the normal and non-normal inclusions between triangle groups.  These are given by concatenations as described by the following


\begin{thm}\label{thm:incltrigroup}
The normal inclusions between hyperbolic triangle groups have the forms
\begin{itemize}
    \item[1.] $\Delta(s,s,t) \lhd_2 \Delta(2,s,2t)$, where $(s-2)(t-1)>2$ and the quotient group is $\Z_2$,
    \item[2.] $\Delta(t,t,t) \lhd_3 \Delta(3,3,t)$, where $t>3$ and the quotient group is $\Z_3$, and
    \item[3.] $\Delta(t,t,t) \lhd_6 \Delta(2,3,2t)$, where $t>3$ and the quotient group is $\mathfrak S_3$.
\end{itemize}
\smallskip
The non-normal inclusions between hyperbolic triangle groups have the forms
\begin{itemize}
    \item[4.] $\Delta(2,7,7) <_{9} \Delta(2,3,7)$,
    \item[5.] $\Delta(3,8,8) <_{10} \Delta (2,3,8)$,
    \item[6.] $\Delta(4,4,5) <_{6} \Delta (2,4,5)$,
    \item[7.] $\Delta(3,3,7) <_{8} \Delta (2,3,7)$,
    \item[8.] $\Delta(t,2t,2t) <_{4} \Delta (2,4,2t)$, with $t\ge3$,
    \item[9.] $\Delta(2,t,2t) <_3 \Delta (2,3,2t)$ with $t\ge4$, and
    \item[10.] $\Delta(3,t,3t) <_{4} \Delta (2,3,3t)$ with $t\ge 3$.
\end{itemize}
where $H \lhd_n G$ means an index $n$ normal subgroup of $G$ and $<_n$ simply means an index $n$ subgroup of $G$.
\end{thm}

\noindent Lower values of the parameters $s,\,t$ correspond to inclusions between euclidean or spherical triangle groups. We conclude by mentioning that the list above is a little redundant because, as the reader may easily check, for $s,t>3$ the third inclusion is a combination of the first and the second. There are some other cases which we have not listed above because they can be deduced by combinations of the inclusions above. These are all listed in \cite[Theorem 3.12]{JW}, compare with \cite[Section \S2]{CV}, and correspond to the four inclusions $\Delta(7,7,7) <_{24} (2,3,7)$, $\Delta(4,8,8)<_{12} \Delta(2,3,8)$, $\Delta(9,9,9) <_{12} \Delta(2,3,9)$ and $\Delta(t,4t,4t) <_6 \Delta(2,3,4t)$ with $t\ge2$.

\smallskip

\noindent \textit{How do we interpret Theorem \ref{thm:incltrigroup} in terms of tessellations?} In order to provide a satisfactory answer, in the first place we recall the following fact from which Proposition \ref{prop:orpres} turns out a special case. Let $\Gamma$ be a group acting properly discontinuously on $\mathbb H^2$ and let $\Gamma'$ be a finite index subgroup with $n$ distinct cosets, say $\Gamma'\gamma_1,\dots,\Gamma'\gamma_n$. If $\mathcal P$ is a fundamental domain for $\Gamma$, then $\mathcal P'=\gamma_1\cdot\mathcal P\,\cup\,\cdots\,\cup\,\gamma_n\cdot\mathcal P$ is a fundamental domain for $\Gamma'$. In particular, it is possible to choose $\gamma_1,\dots,\gamma_n$ so that $\mathcal P'$ turns out a connected polygon. Since $\Gamma$ is any group acting properly discontinuously on $S$; this argument applies in particular to any triangle group $\Delta(p,q,r)$ acting on $\mathbb H^2$. Any subgroup of finite index in $\Delta(p,q,r)$ correspond to a tiling of $\mathbb H$ whose prototile is a polygon obtained as the union of $n$ congruent and adjacent quadrilaterals each of which is a fundamental domain for $\Delta(p,q,r)$.

\begin{ex}[An easy example]

\begin{figure}[!ht]
\centering
\subfloat[][Symmetric tiling associated to the triangle group $\Delta(2,4,6)$]
   {\includegraphics[width=0.425\textwidth]{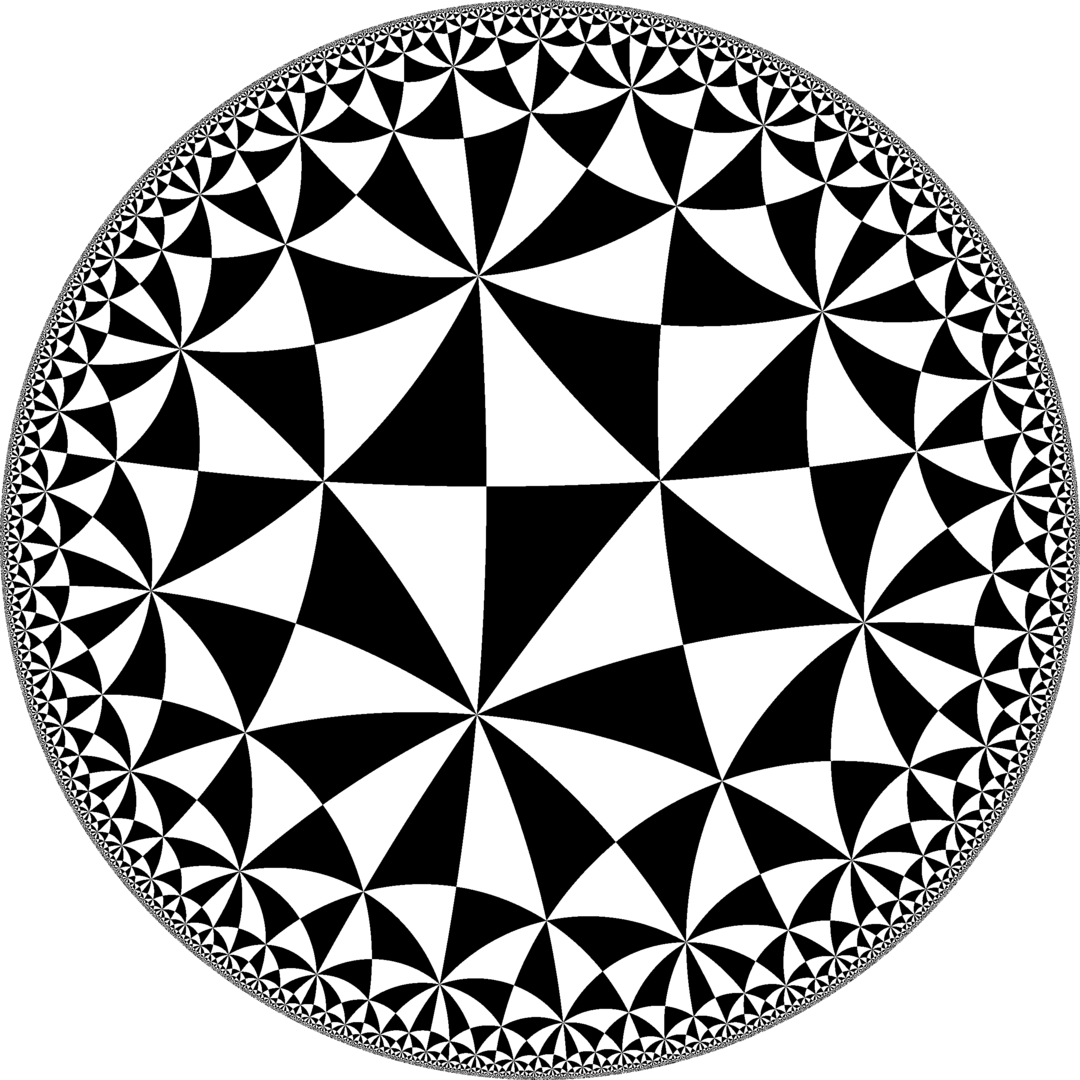}} \quad
\subfloat[][Symmetric tiling associated to the triangle group $\Delta(3,4,4)$]
   {\includegraphics[width=0.425\textwidth]{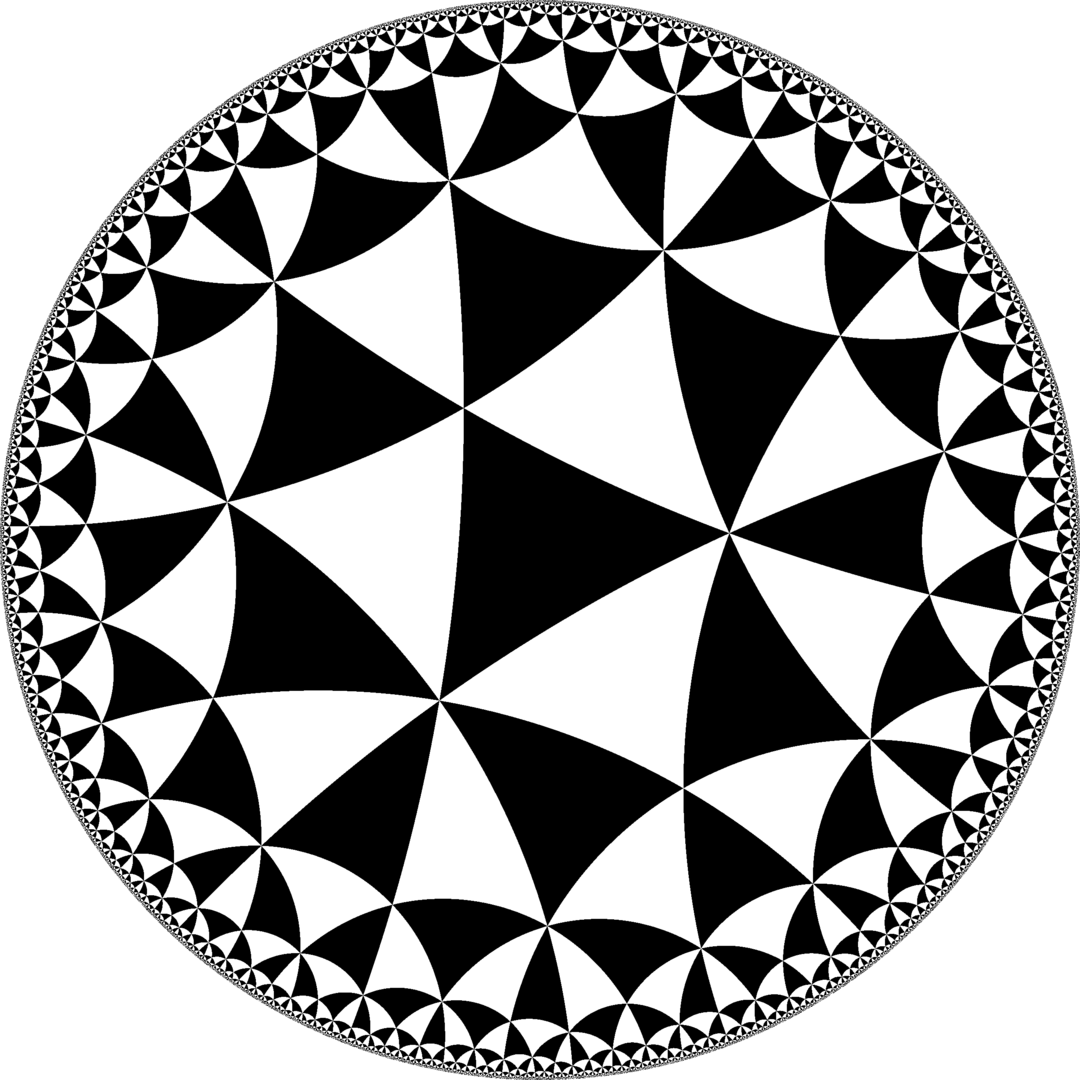}} \\
\caption{On the left the reader can see the symmetric tiling of $\mathbb D$ whose prototile is a hyperbolic triangle with inner angles $\pi/2$, $\pi/4$ and $\pi/6$. On the right the reader can see the symmetric tiling of $\mathbb D$ whose prototile is a hyperbolic triangle with inner angles $\pi/3$, $\pi/4$ and $\pi/4$. The reader is warmly invited to subdivide the tiling on the right in order reproduce the tiling on the left.}
\label{fig:comparison}
\end{figure}

\noindent Consider the triangle group $\Delta(2,4,6)$. It is an easy matter to check that it is of hyperbolic type. A fundamental domain for $W(2,4,6)$-action on $\mathbb H^2$ is a hyperbolic triangle, say $T_1$, with inner angles $\pi/2$, $\pi/4$ and $\pi/6$. Therefore, as a consequence of Proposition \ref{prop:orpres}, a fundamental domain for $\Delta(2,4,6)$ is a quadrilateral obtained by doubling $T_1$. We double $T_1$ along the edge having the vertices with angle $\pi/2$ and $\pi/6$ as extremal points. The resulting quadrilateral is degenerate because it has a vertex with angle $\pi$. In other words, we double $T_1$ so that the resulting shape is a triangle, say $T_2$, having $\pi/4$, $\pi/4$ and $\pi/3$ as inner angles. Clearly, this is a fundamental domain for the action of $W(3,4,4)$ and we finally double $T_2$ in order to obtain a quadrilateral $Q$ as the fundamental domain for $\Delta(3,4,4)$. More precisely, we double $T_2$ along the unique edge having the vertices with angles $\pi/4$ as the extremal points. The reader may now observe such a quadrilateral contains four copies of $T_1$ and admits a symmetry of order two given by a rotation of angle $\pi$ about its center. In particular, the symmetric tiling having $Q$ as a prototile is a sub tessellation of the symmetric tiling having $T_2$ as a prototile. See Figure \eqref{fig:comparison}.
\end{ex}


\section{Platonic solids}\label{platsol}

\noindent In this section we focus on regular tessellations of the sphere and on spherical triangle groups. In subsection \S\ref{ssec:regtess} we have already counted the possible regular tessellations of the sphere. However, our previous approach relies on a standard fact of spherical geometry, namely that the sum of the interior angles of a $p\,-\,$polygon is always greater than $(p-2)\pi$. In fact, the sum of the inner angles of any spherical triangle is bigger than $\pi$ and then, by reasoning as in the Euclidean case, we obtain the claimed lower bound. Our aim here is to provide a more topological approach to show how these tessellations are related with the classical Platonic solids.

\smallskip

\subsection{How many Platonic solids?} The Platonic solids have been known since antiquity and Greeks studied them extensively. Theaetetus was the first to give a mathematical description of Platonic solids and he may have been responsible for the first known proof that no \textit{convex regular polyhedra} exist other than the tetrahedron, cube, octahedron, dodecahedron and icosahedron. With regular here we mean a polyhedron such that its faces are all congruent to a regular $p\,-\,$polygon. A topological proof of this result is based on the following remark known as

\begin{ethm}
For a given convex polyhedron, \textit{i.e.} homeomorphic to a sphere, let $V,\,E,\,F$ be the number of vertices, edges and faces respectively. Then $V\,-\,E\,+\,F\,=\,2$.
\end{ethm}

\noindent There are several proofs of this classic fact, one of which is the following due to by Cauchy. Given a convex polyhedron, remove any of its faces. By pulling the edges of the missing face away from each other, deform all the rest into a planar graph of points and curves, in such a way that the perimeter of the missing face is placed externally, surrounding the graph obtained. Observe that, since our convex polyhedron is homeomorphic to the sphere this is always possible. The number of vertices and edges has remained the same, but the number of faces has been reduced by $1$. Therefore, proving Euler's formula for the polyhedron reduces to proving $V\,-\,E\,+\,F\,=\,1$ for this deformed, planar object. In this planar graph, if there is a face with more than three sides, we draw a diagonal, \textit{i.e.} a curve through the face connecting two vertices that are not yet connected. This adds one edge and one face but it does \textit{not} change the number of vertices. As a consequence, drawing an edge does not alter the quantity $V\,-\,E\,+\,F$. We iterate this process until all the faces are triangles. Next we apply repeatedly one of the following operations so that the exterior boundary is always a simple cycle -- \textit{i.e.} the exterior network bounds a polygon with connected interior.
\begin{itemize}
    \item[1.] Remove a triangle with only one edge adjacent to the exterior. Notice that this decreases both the number of edges and faces by one but it does not change the number of vertices. So $V\,-\,E\,+\,F$ is unaltered by this operation.
    \smallskip
    \item[2.] Remove a triangle with two edges shared by the exterior of the network. Once such a triangle is removed, the number of vertices and faces both decrease by one and the number of edges decreases by two. Therefore, $V\,-\,E\,+\,F$ remains unaltered.
\end{itemize}

\noindent These transformations eventually reduce the planar graph to a single triangle and it is easy to check that $V\,-\,E\,+\,F\,=\,1$ in this case. Hence $V\,-\,E\,+\,F\,=\,1$ holds for the planar graph we have obtained. As a consequence, $V\,-\,E\,+\,F\,=\,2$ for a convex polyhedron and the Euler's observation follows as desired.

\smallskip

\noindent Let us move back on our main goal, we would like to show that only five convex regular polyhedrons exist. Recall that we require all faces to be congruent to a given regular $p\,-\,$polygon. Beyond Euler's observation, for a regular polyhedron we may observe the following relations hold. In the first place we notice that each face is adjacent to $p$ edges and each edge is adjacent to exactly two faces. Hence $p\,F\,=\,2\,E$. On the other hand, each vertex is adjacent to exactly $q$ faces and each face is adjacent to $p$ vertices. Hence $p\,F=q\,V$. Therefore
\begin{equation}
    p\,F\,=\,2\,E\,=\,q\,V.
\end{equation}
\noindent By combining this equation with $V\,-\,E\,+\,F\,=\,2$, we get
\begin{equation}
    \frac{2\,E}{q}\,-\, E\,+\, \frac{2\,E}{p}=2
\end{equation}
that is
\begin{equation}
    \frac{1}{p}+\frac{1}{q}=\frac{1}{2}+\frac{1}{E}.
\end{equation}
\noindent Since $E$ is strictly positive, not surprisingly, we find once again condition \eqref{eq:sss} above
\begin{equation}\label{eq:sss2}
    \frac{1}{p}+\frac{1}{q}>\frac{1}{2}.
\end{equation}
\noindent Since a non-degenerate $p-$regular polygon has at least $3$ edges it follows that $p\ge3$. If we restrict ourselves to $q\ge3$, then \eqref{eq:sss2} has only five solutions, \textit{i.e.} the pairs $\{3,3\}$, $\{3,4\}$, $\{4,3\}$, $\{3,5\}$ and $\{5,3\}$ that we have already found above in \S\ref{ssec:regtess}. Therefore there exactly five Platonic solids as already alluded several times along the course of the present survey. These are the tetrahedron that corresponds to the Schl\"afli symbol $\{3,3\}$, the cube $\{4,3\}$, the octahedron $\{3,4\}$, the dodecahedron $\{5,3\}$ and the icosahedron $\{3,5\}$.  

\smallskip

\noindent Clearly $q$ cannot be one because each edge is shared by two faces and therefore each vertex is adjacent to at least two faces. In principle, $q=2$ is admissible; however, in this special case, the reader may convince themselves that the resulting polyhedron is degenerate as it must have exactly two faces. In other words, the resulting polyhedron is obtained as the identification space of two regular $p-$gons glued edge-by-edge. This is, indeed, a degenerate polyhedron in $\mathbb E^3$ -- the $3-$dimensional space -- and the reader may notice a non-casual similarity between this latter and the dihedron introduced in \S\ref{ssec:regtess}.

\subsubsection*{Duality} There is a well-defined notion of \textit{duality} for regular polyhedrons. Let $\mathfrak P$ be a regular polyhedron with $F$ faces and let $\{p_1,\dots,p_F\}$ be their centers. We join two points $p_i$ and $p_j$ with an edge if and only if their respective faces are adjacent. The resulting space turns out to be the $1-$skeleton of a regular polyhedron, say $\mathfrak Q$ inscribed inside $\mathfrak P$. The reader may observe that if the initial polyhedron has $V$ vertices and $F$ faces, the resulting one has $F$ vertices and $V$ faces. In particular, the same construction applied to $\mathfrak Q$ yields the $1-$skeleton of $\mathfrak P$. In the light of this observation we shall that $\mathfrak P$ and $\mathfrak Q$ are dual of each other. 

\smallskip

\noindent We have seen above that may associate to Platonic solid a pair $\{p,q\}$; \textit{i.e.} its Schl\"afli symbol. Some pairs are evidently symmetrical and these symmetries reflect the fact that some regular polyhedra are dual of each other. For instance, the symmetry between $\{4,3\}$ and $\{3,4\}$ reflects the fact that the cube and the octahedron are the dual of each other. Similarly, the symmetry between the pairs $\{5,3\}$ and $\{3,5\}$ reflects the fact that the icosahedron and the dodecahedron are the dual of each other. Finally the pair $\{3,3\}$ is symmetric and, in fact, the tetrahedron is \textit{self-dual}. 

\smallskip

\subsection{From Platonic solids to tessellations of $\mathbb S^2$} We have seen in the preceding section that there are exactly five platonic solids. From each one of them we shall realise a tiling of $\mathbb S^2$. This means that, if we ignore hosohedrons and dihedrons, there are exactly five regular tessellations of $\mathbb S^2$. If divide each tile into right-angle spherical triangles we obtain a symmetric triangular tilings of $\mathbb S^2$. The group of orientation preserving symmetries of each one of them is a spherical triangle group $\Delta(2,p,q)$. However there are only three of such triangle groups, that is $\Delta(2,3,2)$, $\Delta(2,3,4)$ and $\Delta(2,3,5)$. It necessarily follows that some tessellations share the same symmetry group. Our aim now is to show that if two Platonic solids are dual, then they determine the same tiling into right-angle spherical triangles and hence the induced tilings have the same triangle group as the group of symmetries. 

\smallskip

\noindent Let $\mathfrak P$ be a regular polyhedron, then the sphere $\mathbb S^2$ can be considered as the concentric circumscribed sphere whose centre is the centre of symmetry of $\mathfrak P$. We may imagine such a center of symmetry as an infinitesimal light source. As soon as we turn on the light, the faces of $\mathfrak P$ project to the sphere and hence determine on it a regular tiling having a regular (spherical) polygon as a prototile. It is worth mentioning that such a projection is not conformal because the faces of a regular polyhedron are all congruent to a regular polygon in $\mathbb E^2$ whereas the tiles of the induced tiling on $\mathbb S^2$ are spherical polygons (pairwise congruent). We obtain in this way all possible regular tilings of $\mathbb S^2$.

 \begin{figure}[!ht]
     \centering
     \begin{tikzpicture}[line cap=round,line join=round,3d view={40}{35}]
\pgfmathsetmacro\ph{(1+sqrt(5))/2} 
\pgfmathsetmacro\ed{2}             
\pgfmathsetmacro\hh{\ed*\ph}       
\coordinate (A1) at ( \hh,-\ed,  0);
\coordinate (B1) at ( \hh, \ed,  0);
\coordinate (C1) at (-\hh, \ed,  0);
\coordinate (D1) at (-\hh,-\ed,  0);
\coordinate (A2) at ( \ed,  0,-\hh);
\coordinate (B2) at (-\ed,  0,-\hh);
\coordinate (C2) at (-\ed,  0, \hh);
\coordinate (D2) at ( \ed,  0, \hh);
\coordinate (A3) at (  0, \hh,-\ed);
\coordinate (B3) at (  0, \hh, \ed);
\coordinate (C3) at (  0,-\hh, \ed);
\coordinate (D3) at (  0,-\hh,-\ed);
\mytriangle{A2}{A1}{D3}{0.4};
\mytriangle{A1}{B1}{A2}{0.5};
\mytriangle{D1}{C3}{D3}{0.2};
\mytriangle{A1}{D3}{C3}{0.3};
\mytriangle{C3}{D2}{A1}{0.0};
\mytriangle{B1}{A1}{D2}{0.1};
\mytriangle{D2}{B3}{B1}{0.3};
\mytriangle{C2}{D1}{C3}{0.2};
\mytriangle{D2}{C3}{C2}{0.1};
\mytriangle{B3}{C2}{D2}{0.2};

\end{tikzpicture}
     \caption{Icosahedron with tiled faces. Each tile of each face has a "kite" shape obtained by gluing two right-angle Euclidean triangle with angles $\big( \frac\pi2,\,\frac\pi3,\,\frac\pi6\big)$. We may count exactly $60$ kites on the icosahedron, in fact $3$ for each face. The symmetry group of acts transitively, \textit{i.e.} any pair of kites are related by a unique symmetry of the icosahedron. As we shall see below, its symmetry group identifies with $\Delta(2,3,5)\cong A_5$ -- the alternating group of five elements which has order $60$. 
     }
     \label{fig:regpolico}
 \end{figure}
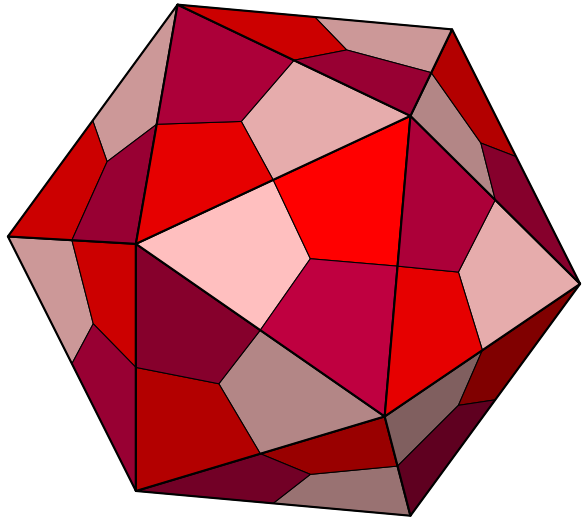
 
 \smallskip

\noindent Although there are five regular tilings, there are only three symmetric triangular tilings. Given a regular polyhedron, we can consider its barycentric subdivision into smaller triangles, see Figure \eqref{fig:regpolico}. Since all faces are all isometric to a regular polygon in $\mathbb E^2$, all the triangles are pairwise congruent to a certain right-angle Euclidean triangle in $\mathbb E^2$. As a consequence, as soon as we turn on the light source, the subdivided faces project to a symmetric triangular tessellation of $\mathbb S^2$. Unlike symmetric triangular tessellations of $\mathbb E^2$ and $\mathbb H^2$, any symmetric triangular tessellation of $\mathbb S^2$ has finitely many tiles. In fact, let $\mathcal T$ be a symmetric triangular tessellation of $\mathbb S^2$ with prototile $t$. Recall that, as the sphere is compact, $t$ is a compact triangle with finite volume $\text{vol}(\,t\,)$. Since the volume of the unit sphere is equal to $4\pi$, and $\mathcal T$ is symmetric -- this is a crucial assumption -- we obtain that $\mathcal T$ has 
\begin{equation}
    |\,\mathcal T\,|=\frac{4\pi}{\text{vol}(\,t\,)}
\end{equation}
tiles, where $|\,\cdot\,|$ denotes the cardinality of $\mathcal T$. This is far from being specific to triangular tilings and, in fact, we have the following

\begin{prop}\label{prop:counttile}
Let $\mathcal T$ be a symmetric tessellation of $\mathbb S^2$ with a polygon $t$ as a prototile in $\mathbb S^2$. Then 
\[ |\,\mathcal T\,|=\frac{4\pi}{\textnormal{vol}(\,t\,)}.
\]
\end{prop}

\noindent The proof is substantially identical to the argument above. We can therefore count the number of triangular tiles of a given symmetric triangular tessellation. In what follows, given a triangle $t$ and a vertex $v$ of $t$ we define the \textit{valence} of $v$ as the number of triangles adjacent to $v$. By keeping this definition in mind, for each regular polyhedron we count the triangles of the barycentric subdivision of their faces and the valences of their vertices.

\begin{itemize}
    \item[1.] \textit{Tetrahedron} has $4$ faces all congruent to an equilateral triangle. Hence the barycentric subdivision has $24$ right-angle triangles   with inner angles $\pi/2$, $\pi/3$ and $\pi/6$. Each vertex of the resulting subdivision has valence $6$, that is each vertex is adjacent to $6$ triangles.
    \smallskip
    \item[2.] \textit{Cube} has $6$ faces all congruent to a square. Hence the barycentric subdivision has $48$ right-angle triangles with inner angles $\pi/2$, $\pi/4$ and $\pi/4$. In this case, the resulting subdivision has $6$ vertices with valence $8$ and $8$ vertices with valence $6$.
    \smallskip
    \item[3.] \textit{Octahedron} has $8$ faces all congruent to an equilateral triangle. Hence the barycentric subdivision has $48$ right-angle triangles with inner angles $\pi/2$, $\pi/3$ and $\pi/6$. Like the case above, the resulting barycentric subdivision has $6$ vertices with valence $8$ and $8$ vertices with valence $6$.
    \smallskip
    \item[4.] \textit{Dodecahedron} has $12$ faces all congruent to a pentagon. Hence the barycentric subdivision has $120$ right-angle triangles with inner angles $\pi/2$, $3\pi/10$ and $\pi/5$. In this case, the resulting subdivision has $12$ vertices with valence $10$ and $20$ vertices with valence $6$. Finally
    \smallskip
    \item[5.] \textit{Icosahedron} has $20$ faces all congruent to an equilateral triangle. Hence the barycentric subdivision has $120$ right-angle triangles with inner angles $\pi/2$, $\pi/3$ and $\pi/6$. Like the dodecahedron, the resulting barycentric subdivision has $12$ vertices with valence $10$ and $20$ vertices with valence $6$.
\end{itemize}
 
 \noindent According to the list above, we can deduce the following
 
 \begin{prop}
 Each Platonic solid admits a triangulation into right-angle triangles determined by barycentric subdivisions of its faces. In particular, if two Platonic solids are one the dual of the other then they admit a triangulation with the same number of triangular tiles whose prototiles have the same valences at the vertices.
 \end{prop}
 
 \noindent Since the projection clearly does not alter the number of tiles and valence of the vertices, it follows that the spherical tessellations induced by the projection will have the same number of tiles and the vertices have the same valences. The key fact is the following
 
 \begin{prop}
 Let $\mathcal T_1$ and $\mathcal T_2$ be two symmetric triangular tessellations of $\mathbb S^2$ with the same number of tiles. Then $\mathcal T_1=R\cdot\mathcal T_2$, for some $R\in\textnormal{Iso}(\mathbb S^2)$, if and only if their respective prototiles have the same valences at the vertices.
 \end{prop}
 
 \begin{proof}
 Let $t_1$ and $t_2$ be the prototiles of $\mathcal T_1$ and $\mathcal T_2$ respectively.  Proposition \ref{prop:counttile} applies because both tilings are symmetric and hence
 \begin{equation} \frac{4\pi}{\textnormal{vol}(\,t_1\,)}=|\,\mathcal T_1\,|=|\,\mathcal T_2\,|=\frac{4\pi}{\textnormal{vol}(\,t_2\,)}.
 \end{equation}
 As a consequence $\text{vol}(\,t_1\,)=\text{vol}(\,t_2\,)$. Since the prototiles have the same valences at the vertices then they have the same inner angles. Therefore, $t_1=R\cdot t_2$ and hence $\mathcal T_1\,=\,R\cdot\mathcal T_2$ for some $R\in\textnormal{Iso}(\mathbb S^2)$.
 \end{proof}
 
 \noindent As a consequence we have the following
 
\begin{thm}
A tetrahedron determines a symmetric tiling of $\mathbb S^2$ with prototile a spherical right-angle triangle with angles $\big(\pi/2,\,\pi/3,\,\pi/3)$. The cube and the octahedron determine the same symmetric tiling of $\mathbb S^2$ with prototile a spherical right-angle triangle with angles $\big(\pi/2,\,\pi/3,\,\pi/4)$. Finally, the dodecahedron and the icosahedron determine the same symmetric tiling of $\mathbb S^2$ with prototile a spherical right-angle triangle with angles $\big(\pi/2,\,\pi/3,\,\pi/5)$.
\end{thm}

\noindent It remains to determine the groups of symmetries of these triangular tilings of $\mathbb S^2$. By keeping in mind our previous discussion in \S\ref{ssec:tritiles} and \S\ref{tg2} we have the following

\begin{cor}
Let $\mathcal T$ be the symmetric triangular tiling of the sphere determined by the barycentric subdivision of the tetrahedron. Then $\text{Sym}^+(\mathcal T)\cong\Delta(2,3,3)$. Similarly, let $\mathcal T$ be the symmetric triangular tiling of the sphere determined by the barycentric subdivision of the cube or the octahedron. Then $\text{Sym}^+(\mathcal T)\cong\Delta(2,3,4)$. Finally, let $\mathcal T$ be the symmetric triangular tiling of the sphere determined by the barycentric subdivision of the dodecahedron or the icosahedron. Then $\text{Sym}^+(\mathcal T)\cong\Delta(2,3,5)$.
\end{cor}

 \begin{figure}[!ht]
  \centering
  \includegraphics[scale=0.35]{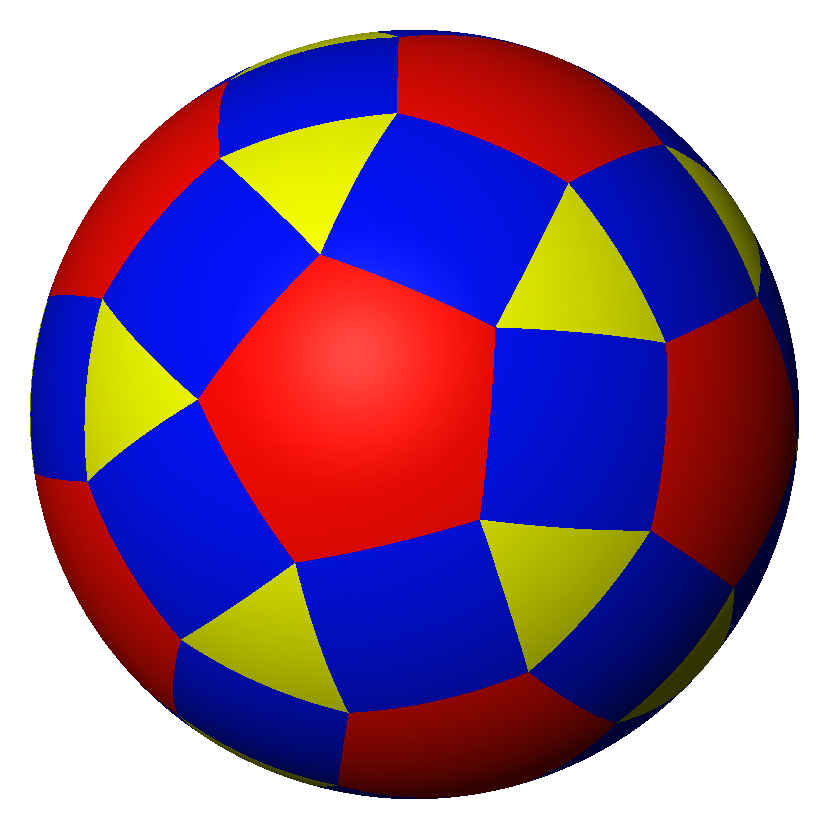}
  \caption{Rhombitrihexagonal tessellation of the sphere. This is an example of semi-regular tessellation that arises from a semi-regular Platonic solid.}
  \label{srhomb}
\end{figure}
 
 \subsection{A short digression to semi-regular Platonic solids}\label{ssec:semiplatsol} We conclude the present Section by spending a few words on semi-regular polyhedra. A semi-regular polyhedron a is a convex polyhedron such that all faces are all regular polygons in $\mathbb E^2$ but they are no longer pairwise congruent. Of course Platonic solids fall into this larger category. However, as soon as the faces are no longer required to be congruent to a given polygon in $\mathbb E^2$, we have plenty of possibilities. In fact, infinitely many. Semi-regular polyhedra include Archimedean solids, prisms and antiprisms. Archimedean solids form a finite family of convex polyhedra, but prisms and antiprisms are both infinite families. By projecting the faces of these semi-regular polyhedra as done above for Platonic solids we obtain semi-regular tilings of $\mathbb S^2$. An example of such a tiling is depicted in Figure \eqref{srhomb} and the reader may observe that it is the spherical counterpart of the Rhombitrihexagonal tessellation of the Poincar\'e disc.
 
\bibliographystyle{amsalpha}
\bibliography{tess}

\end{document}